\documentclass[12pt,reqno]{amsart}
\usepackage[letterpaper,margin=1.2in]{geometry}

\usepackage{amsmath,amssymb,amsthm,verbatim, xcolor}
\usepackage{hyperref}

\newtheorem{theorem}[subsection]{Theorem}
\newtheorem{lemma}[subsection]{Lemma}
\newtheorem{cor}[subsection]{Corollary}
\newtheorem{prop}[subsection]{Proposition}

\newtheorem{ass}[subsection]{Assumptions}

\theoremstyle{definition}
\newtheorem{definition}[subsubsection]{Definition}
\newtheorem{remark}[subsection]{Remark}
\newtheorem{example}[subsection]{Example}

\newcommand{\haus}{\mathcal{H}}

\newcommand{\spt}{\mathrm{spt}}

\newcommand{\reg}{\mathrm{reg}}

\newcommand{\graph}{\mathrm{graph}}
\newcommand{\tr}{\mathrm{tr}}

\newcommand{\eps}{\epsilon}

\newcommand{\sing}{\mathrm{sing}}
\newcommand{\cF}{\mathcal{F}}
\newcommand{\R}{\mathbb{R}}
\newcommand{\bC}{\mathbf{C}}

\newcommand{\del}{\partial}

\newcommand{\cG}{\mathcal{G}}
\newcommand{\geucl}{g_{eucl}}

\newcommand{\cH}{\mathcal{H}}

\newcommand{\cE}{\mathcal{E}}

\title[Symmetric (log-)epiperimetric]{The symmetric (log-)epiperimetric inequality and a decay-growth estimate}
\author{Nick Edelen, Luca Spolaor, Bozhidar Velichkov}

\numberwithin{equation}{section}

\begin{document}

\begin{abstract}
We introduce a symmetric (log-)epiperimetric inequality, generalizing the standard epiperimetric inequality, and we show that it implies a growth-decay for the associated energy: as the radius increases energy decays while negative and grows while positive.  One can view the symmetric epiperimetric inequality as giving a log-convexity of energy, analogous to the 3-annulus lemma or frequency formula.  We establish the symmetric epiperimetric inequality for some free-boundary problems and almost-minimizing currents, and give some applications including a ``propagation of graphicality'' estimate, uniqueness of blow-downs at infinity, and a local Liouville-type theorem.
\end{abstract}

\maketitle

\section{Introdution}\label{sec:intro}

We consider the decay-growth properties of certain functions $u$ defined on $\R^n$.  Typically when $u$ solves (or almost solves) a second-order elliptic PDE, $u$ will admit a Fourier-type decomposition and/or a monotone frequency function and/or a 3-annulus-type lemma.  Any of these imply a decay-growth estimate for a suitably scaled $u_r := r^{-m} u(r \cdot)$ (or, often, the scale-invariant radial derivative $r \del_r u_r$): there is a threshold $r_0$ so that, as the radius $r$ increases, the norm $||u_r||_{L^2(\del B_1)}$ will exhibit polynomial or logarithmic decay in $r$ while $r \leq r_0$, and polynomial or logarithmic growth in $r$ while $r \geq r_0$.

In this paper we introduce a \emph{symmetric (log-)epiperimetric inequality} for an energy functional $\cE: H^1(B_1) \to \R$, which encompasses both the usual epiperimetric inequality and a new ``reverse'' epiperimetric inequality.  We show that when $u$ and $\cE$ satisfy the symmetric epiperimetric inequality and a fairly universal monotonicity-type formula, then both $\cE(u_r)$ and $r \del_r u_r$ exhibit a decay-growth estimate as above.  In fact we will characterize the threshold radius $r_0$ as the radius at which the associated energy $\cE(u_r)$ changes sign.

As an application, we prove the symmetric epiperimetric inequality for some free-boundary-type variational problems and for almost-minimizing currents.  The main consequence is a ``quantitative uniqueness of tangent cone'' which allows one to propagate graphicality or closeness near a smooth cone in regions where the energy stays close (see Theorems \ref{thm:ac-param}, \ref{cor:o-unique}, \ref{thm:am-param}, and Example \ref{ex:ac}).  This implies uniqueness of blow-ups and blow-downs, but is particularly useful in blow-up arguments to control transition regions between scales.  For example, using our estimates we can prove some ``local Liouville'' theorems for minimizers of the Alt-Caffarelli functional and almost-minimizing currents (Theorems \ref{thm:ac-lious}, \ref{thm:am-lious}).

Epiperimetric-type inequalities have already found extensive use in free-boundary type problems (\cite{We}, \cite{SpVe}, \cite{cospve1}, \cite{EnSpVe}, \cite{SpVe2}) and minimal surface theory (\cite{Re}, \cite{Ta}, \cite{Wh}, \cite{EnSpVe2}) to prove uniqueness of blow up and decay estimates near singularities.  A key advantage of the symmetric epiperimetric inequality is that it can be applied in an annulus \emph{without} a priori knowledge of the singular behavior (compare \cite[Appendix 1]{Ed}).  

\vspace{3mm}


We observe here that many solutions $u \in H^1(B_R \subset \R^n)$ to variational problems involving an energy $\cE : H^1(B_1) \to [-1, 1]$ satisfy the following two properties:
\begin{enumerate}
\item \emph{Monotonicity}: writing $u_r(x) = r^{-m} u(rx)$ and $z_r$ for the $m$-homogenous extension of $u_r$, then for a.e. $r \in (\rho, R)$:
\begin{equation}\label{eqn:intro-mono}
\frac{d}{dr} \cE(u_r) \geq \frac{c_E}{r}(\cE(z_r) - \cE(u_r)) + \frac{c_E}{r} \int_{\del B_1} r^2 |\del_r u_r|^2 d\haus^{n-1}
\end{equation}
for some fixed $c_E > 0$, $m \in (2-n, \infty)$, $\rho \in [0, R]$.
\item \emph{Symmetric log-epiperimetric inequality}: for a.e. $r \in (\rho, R)$:
\begin{equation}\label{eqn:intro-epi}
\cE(u_r) \leq \cE(z_r) - \eps |\cE(z_r)|^{1+\gamma},
\end{equation}
for some $\gamma \in [0, 1), \eps \in (0, 1]$.
\end{enumerate}
When $\cE(z_r) \geq 0$, then \eqref{eqn:intro-epi} reduces to the usual (log-)epiperimetric inequality.  However, unlike standard (log-)epiperimetric, because of the absolute values, \eqref{eqn:intro-epi} can control $\cE(u_r)$ even when $\cE(z_r) \leq 0$.  We call \eqref{eqn:intro-epi} ``symmetric'' because it does not care about the sign of $\cE(z_r)$, or the direction of $r$ (increasing or decreasing).  For comparison, we might call \eqref{eqn:intro-epi} with $(-\cE(z_r))$ in place of $|\cE(z_r)|$ a ``reverse'' epiperimetric-type inequality.

We show in Theorem \ref{thm:main} that the above two properties guarantee a decay-growth estimate, which stated informally says the following.
\begin{theorem}\label{thm:teaser}
In the notation of the above, as $r$ increases from $\rho$ to $R$, then $\cE(u_r)$ is increasing, and both $\cE(u_r)$ and $||r \del_r u_r||_{L^2(\del B_1)}$ polynomially/logarithmically decay while $\cE(u_r) \leq 0$, and polynomially/logarithmically grow while $\cE(u_r) \geq 0$.  The nature of growth (polynomial or logarithmic) depends on whether $\gamma = 0$ or $\gamma > 0$ (resp.).

If $\rho = 0$ (resp. $R = \infty$), then $u_r$ has an $L^2(\del B_1)$ limit as $r \to 0$ (resp. $r \to \infty$).  If $\rho = 0$ \emph{and} $R = \infty$, then $\del_r u_r \equiv 0$.
\end{theorem}

Let us illustrate with an example.  In many free-boundary-type problems the associated energy of interest is a modification of the following Weiss-type energy
\begin{equation}\label{eqn:weiss-def}
W_0^m(u) = \int_{B_1} |Du|^2 - m \int_{\del B_1} u^2 , \quad u \in H^1(B_1).
\end{equation}
It is a now-standard computation due to Weiss \cite{We2} that if $u_r(x) = r^{-m}u(rx)$, then $r \mapsto W_0^m(u_r)$ is absolutely continuous in $r \in (0, 1]$, and
\begin{equation}\label{eqn:weiss}
\frac{d}{dr} W_0^m(u_r) = \frac{n+2m-2}{r}(W_0(z_r) - W_0(u_r)) + \frac{1}{r} \int_{\del B_1} r^2|\del_r u_r|^2,
\end{equation}
where $z_r(x) = |x|^m u_r(x/|x|)$ is the $m$-homogenous extension of $u_r$.

\begin{example}\label{ex:ac}
We consider here the Alt-Caffarelli funtional $J$ defined on an open set $U$  by
\begin{equation}\label{eqn:ac}
J(u, U) = \int_{U} |Du|^2 dx + |\{ u > 0 \}|, \quad u \in H^1(U).
\end{equation}
We say that a nonnegative function $u \in H^1(U)$ minimizes $J$ if $J(u, U) \leq J(v, U)$ for all $u - v \in H^1_0(U)$, and $u \in H^1_{loc}(U)$ locally-minimizes $J$ if $J(u, W) \leq J(v, W)$ for all $u - v \in H^1_0(W)$ for some $W \subset\subset U$.  If $U = B_1$ we simply write $J(u)$.

The functional $J$ was introduced by \cite{AlCa}, who proved existence, compactness, and regularity properties of minimizers and their free-boundaries $\del \{ u > 0 \} \cap U$.  Minimizers $u$ are known to be locally-Lipschitz, and (from additional work of \cite{We2}, \cite{JeSa}) the free-boundary $\del \{ u > 0 \} \cap U$ is known to be analytic away from a singular set of dimension at most $n-d^*$, for some critical dimension (as yet unknown) $d^* \in \{ 5, 6, 7\}$.

The natural scaling for $J$ is $1$-homogeneous, so let us write in this Example $u_r(x) = r^{-1}u(r x)$ and $\xi_r(x) = r^{-1} \xi(r x)$.  It follows from the Weiss monotonicity \eqref{eqn:intro-ac-mono} that any blow-up (or blow-down) $u_0$ is $1$-homogenous.  \cite{AlCa} proved a regularity theory that shows $\del \{ u > 0 \}$ is regular at points where the blow-up is $u_0(x) = (x \cdot e)_+$ for some unit vector $e$.  \cite{EnSpVe} used a log-epiperimetric inequality to prove uniqueness of blow-ups $u_0$ with isolated singularities (such as in dimension $d^*$).

\vspace{3mm}

Let us fix in this Example a $1$-homogenous minimizer $u_0$ of $J$ which is regular away from $0$.  We shall demonstrate the nature and use of the symmetric log-eperimetric inequality for minimizers near $u_0$, which is covered in more detail in Section \ref{sec:ac}.  To study the behavior of $u$ near a $1$-homogenous $u_0$, the natural energy to consider is a normalized Weiss energy
\[
\cE(u) = W(u) - W(u_0), \quad W(u) = W^1_0(u) + | \{ u > 0 \} \cap B_1 |, 
\]
where $W^1_0$ as in \eqref{eqn:weiss-def}.  The Weiss monotonicity \eqref{eqn:weiss} and the coarea formula imply that for any $u \in H^1(B_1)$, $r \mapsto \cE(u_r)$ is absolutely continuous on $r \in (0, 1)$ and satisfies
\begin{equation}\label{eqn:intro-ac-mono}
\frac{d}{dr} \cE(u_r) = \frac{n}{r}(\cE(z_r) - \cE(u_r)) + \frac{1}{r} \int_{\del B_1} r^2 |\del_r u_r|^2  \quad \text{ for } a.e. r \in (0, 1),
\end{equation}
where $z_r(x) = |x|u_r(x/|x|)$ is the $1$-homogenous extension of $u_r$.  Note in particular if $u \in H^1(B_1)$ minimizes $J$, then $r \mapsto \cE(u_r)$ is increasing on $[0,1]$.

\vspace{3mm}

Write $u_0(x = r\theta) = r\sigma(\theta)$, $\Omega = \{ \sigma > 0 \} \subset B_1$, and $C\Omega = \{ u_0 > 0\}$, and $\nu(x)$ for the outer normal of $\Omega \subset B_1$.  Given $\xi : \del\Omega \to \R$, with $||\xi||_{C^0}$ sufficiently small depending only on $\Omega$, we can define a new domain $\Omega_\xi$ as a perturbation of $\Omega$, with boundary given by
\[
\del \Omega_\xi = \{ \cos(\xi(\theta)) \theta + \sin(\xi(\theta)) \nu(\theta) : \theta \in \del\Omega \}.
\]
If $\xi : \Omega \cap A_{R, \rho} \to \R$, then we define conical graph
\[
G_{C\Omega}(\xi) \cap A_{R, \rho} = \cup \{ r \Omega_{\xi_r} : r \in (\rho, R) \}.
\]
Here $A_{R, \rho}(x) = B_R(x) \setminus \overline{B_\rho(x)}$.  Write $\lambda_i^\Omega$ for the Dirichlet eigenvalues of $\Omega$, and $\phi_i^\Omega$ for the corresponding $L^2(\Omega)$-ON eigenfunctions.  We prove in Section \ref{sec:ac} the following. Recall that $\cE(u) = W(u) - W(u_0) \equiv W(u) - W(r\sigma)$.

\begin{theorem}[Symmetric log-epi for Alt-Caffarelli near smooth cones]\label{thm:ac-epi}
There are constants $\delta(\sigma), \eps(\sigma) > 0$, $\gamma(\sigma) \in [0, 1)$, so that the following holds.

Let $\xi \in C^{2,\alpha}(\del\Omega)$, $z \in H^1(\del B_1)$, be such that $\{ z > 0 \} = \Omega_\xi$ and
\begin{equation}\label{eqn:ac-epi-hyp}
||z - \sigma||_{H^1(\del B_1)} < \delta, \quad ||\xi||_{2,\alpha} < \delta.
\end{equation}
Then there is an $h \in H^1(B_1)$ with $h|_{\del B_1} = z$, so that
\begin{equation}\label{eqn:ac-epi-concl}
\cE(h) - \cE(rz) \leq -\eps |\cE(rz)|^{1+\gamma}.
\end{equation}
\end{theorem}

Theorem \ref{thm:ac-epi} implies that if $u \in H^1(B_1)$ minimizes $J$ and satisfies
\begin{gather}
\sup_{r \in (\rho, 1)} ||u_r - u_0||_{H^1(\del B_1)} \leq \delta, \quad \text{ and } \label{eqn:intro-ac-ass1} \\
\del \{ u > 0 \} \cap A_{1, \rho} = G_{C\Omega}(\xi) \cap A_{1, \rho}, \quad \sup_{r \in (\rho, 1)} ||\xi_r||_{C^{2,\alpha}(\Omega)} \leq \delta, \label{eqn:intro-ac-ass2}
\end{gather}
for some $\rho \in [0, 1/2]$, then
\begin{gather}\label{eqn:intro-ac-epi}
\cE(u_r) \leq \cE(z_r) - \eps |\cE(z_r)|^{1+\gamma} \quad \forall r \in (\rho, 1).
\end{gather}
So for a.e. $r \in (\rho, 1)$, $\cE(u_r)$ satisfies both a monotonicity formula \eqref{eqn:intro-ac-mono} like \eqref{eqn:intro-mono} and a symmetric log-epiperimetric inequality \eqref{eqn:intro-ac-epi} like \eqref{eqn:intro-epi}.  Our growth-decay Theorem \ref{thm:main} immediately implies:
\begin{theorem}[Growth-Decay estimates near smooth cones]\label{eqn:ac-main}
There are $c(\sigma), \beta(\sigma) > 0$ so that if $u \in H^1(B_1)$ minimizes $J$ and satisfies \eqref{eqn:intro-ac-ass1}, \eqref{eqn:intro-ac-ass2} for some $\rho \in [0, 1/2]$ (and $\delta, \gamma$ as in Theorem \ref{thm:ac-epi}), and $\rho' \in [\rho, 1]$ is chosen so that $\cE(u_r) \leq 0$ for $r \in [\rho, \rho']$ and $\cE(u_r) \geq 0$ for $r \in [\rho', 1]$, then we have the decay-growth estimates
\begin{align}
&\int_{\max\{s, \rho'\}}^{\max\{r, \rho'\}} ||\del_ru_r||_{L^2(\del B_1)} dr \leq \left\{ \begin{array}{l l} c ( \left[\cE(1)_+\right]^{\frac{\gamma-1}{2}} + \delta \gamma \log(1/r)^{\frac{1-\gamma}{2\gamma}})^{-1}  & \gamma > 0 \\ c \left[\cE(1)_+\right]^{\frac{1}{2}} r^{\frac{\beta}{2}} & \gamma = 0 \end{array} \right. , \label{eqn:intro-ac-grow} \\
&\int_{\min\{s, \rho'\}}^{\min\{r, \rho'\}} ||\del_r u_r||_{L^2(\del B_1)} dr \leq \left\{ \begin{array}{l l} c ( \left[\cE(\rho)_-\right]^{\frac{\gamma-1}{2}} + \delta \gamma \log(s/\rho)^{\frac{1-\gamma}{2\gamma}})^{-1}  & \gamma > 0  \\ c  \left[\cE(\rho)_-\right]^{\frac{1}{2}} (\rho/s)^{\frac{\beta}{2}} & \gamma = 0 \end{array} \right. ,\label{eqn:intro-ac-decay}
\end{align}
for every $\rho \leq s < r \leq 1$, and hence we have the Dini estimate
\begin{equation}\label{eqn:intro-ac-dini}
\int_{\rho}^{1} ||\del_r u_r||_{L^2(\del B_1)} dr \leq c( \left[\cE(\rho)_-\right]^{\frac{1-\gamma}{2}} + \left[\cE(1)_+\right]^{\frac{1-\gamma}{2}}) .
\end{equation}
\end{theorem}

A key point is that no sign of $\cE$ is assumed anywhere, and so no a priori assumptions on the singular set in $B_{\rho}$ are required, but on the other hand one may have both decay and growth in $r$ depending on the sign of $\cE$.  \eqref{eqn:intro-ac-grow}, \eqref{eqn:intro-ac-decay} recover the uniqueness of blow-ups with isolated singularities as proven by \cite{EnSpVe}, and also establish uniqueness of the same kinds of blow-downs (Corollary \ref{cor:ac-unique}).  Additionally, when combined with the $\eps$-regularity of \cite{AlCa}, \eqref{eqn:intro-ac-dini} implies $u$ stays close to $u_0$ while the energy $|\cE|$ stays small.  Precisely, in Theorem \ref{thm:ac-param}, we show that if $u$ satisfies \eqref{eqn:intro-ac-ass1}, \eqref{eqn:intro-ac-ass2} with $\rho = 1/2$ and $\delta(\eps, \sigma)$ sufficiently small, and we have
\[
|\cE(u_r)| \leq \delta(\eps, \sigma) \quad \forall r \in (\rho''/2, 1)
\]
for some $\rho'' \in [0, 1/2]$, then we can \emph{deduce} \eqref{eqn:intro-ac-ass1}, \eqref{eqn:intro-ac-ass2} holds with $\rho''$ in place of $\rho$ and $\eps$ in place of $\delta$.  In Theorem \ref{thm:ac-lious} we show how this ``propogation of closeness'' can be used in blow-up arguments to prove local structure from global rigidity properties.
\end{example}

\vspace{2mm}

In Section \ref{sec:decay} we prove an abstract decay-growth estimate, which shows how the monotonicity formula \eqref{eqn:intro-mono} and the symmetric epiperimetric inequality \eqref{eqn:intro-epi} imply a differential inequality for $\cE$.  In Section \ref{sec:fb} we explain how to prove and apply the symmetric epiperimetric inequality to the obstacle problem, and in Section \ref{sec:am} we adapt the symmetric epiperimetric inequality to almost-minimizing currents.  The setups and applications are similar to the Example above, but we want to highlight that we can additionally handle almost-minimizing currents defined only in an annulus (Section \ref{ssec:am}), similar to the situation considered in \cite{MaNo}.

\textbf{Acknowledgements:} N.E. was supported in part by NSF grant DMS-2204301. L.S. acknowledges the support of the NSF Career Grant DMS 2044954. B.V. was supported by the European Research Council's (ERC) project n.853404 ERC VaReg - {\it Variational approach to the regularity of the free boundaries}, financed by Horizon 2020.

\section{An abstract decay-growth theorem}\label{sec:decay}

In this section we prove that the symmetric (log-)epiperimetric inequality, combined with a monotonicity formula, implies a decay/growth estimate.  We allow errors which polynomially grow or decay.

We shall work in $\R^n$.  Fix $m \in (2-n, \infty)$.  Take $r_1 < r_3$, and $u \in H^1(B_{r_3})$.  For $r \leq r_3$ define the $m$-homogenously rescaled $u_r(x) = r^{-m} u(rx) \in H^1(B_1)$, and for a.e. $r \in (0, r_3)$ define $z_r(x) = |x|^m u_r(x/|x|) \in H^1(B_1)$ to be the $m$-homogenous extension of $u_r$.

We shall assume, throughout this section, that the following assumptions hold.

\begin{ass}\label{ass_symepi}There are constants $\Lambda_\pm \geq 0$, $\alpha, c_E, \eps \in (0, 1]$, $\gamma \in [0, 1)$, and an energy functional $\cE : H^1(B_1) \to [-1, 1]$, so that the following hold:
\begin{enumerate}
\item Almost-monotonicity: $r \mapsto \cE(u_r)$ is a $BV_{loc}$ function on $(r_2, r_3)$, and satisfies
\begin{equation}\label{eqn:main-hyp1}
\frac{d}{dr} \cE(u_r) \geq \frac{c_E}{r}(\cE(z_r) - \cE(u_r)) + \frac{c_E}{r} \int_{\del B_1} r^2 |\del_r u_r|^2 d\haus^{n-1} - \Lambda_+ r^{\alpha - 1} - \Lambda_- r^{-\alpha-1}
\end{equation}
for a.e. $r \in (r_1, r_3)$.

\item Symmetric (log-)epipimetric + almost-minimizing: for a.e. $r \in (r_2, r_3)$ we have
\begin{equation}\label{eqn:main-hyp2}
\cE(u_r) \leq \cE(z_r) - \eps|\cE(z_r)|^{1+\gamma} + \Lambda_+ r^{\alpha} + \Lambda_- r^{-\alpha}
\end{equation}
\end{enumerate}
\end{ass}


Our main Theorem is the following decay-growth estimate for $r\del_r u_r$, which in turn (see Remark \ref{rem:diff-ineq}) controls the differences $||u_r - u_s||_{L^2(\del B_1)}$.
\begin{theorem}[Decay-growth of $r\del_r u_r$]\label{thm:main}
Under the assumptions \eqref{eqn:main-hyp1}, \eqref{eqn:main-hyp2}, and provided $\max\{ \Lambda_+ r_3^{\alpha}, \Lambda_- r_1^{-\alpha} \} \leq 1$, if we define
\begin{equation}
G(r) = \cE(u_r) + 3\alpha^{-1} \Lambda_+ r^{\alpha} - 3\alpha^{-1} \Lambda_- r^{-\alpha}
\end{equation}
then the following holds.

\begin{enumerate}
\item $G$ is increasing, and for a.e. $r \in (r_1, r_3)$ satisfies
\begin{equation}\label{eqn:main-concl1}
G' \geq (\delta/r)|G|^{1+\gamma} ,
\end{equation}
for $\delta = \delta(c_E, \eps, \gamma, \alpha)$.

\item We have the Dini estimate
\begin{equation}\label{eqn:main-concl2}
\int_{r_1}^{r_3} ||r \del_r u_r||_{L^2(\del B_1)} \frac{dr}{r} \leq c( \left[G(r_1)_-\right]^{\frac{1-\gamma}{2}} + \left[G(r_3)_+\right]^{\frac{1-\gamma}{2}}),
\end{equation}
where $v_+ = \max\{ v, 0\}$, $v_- = -\min\{ v, 0\}$, and $c = c(c_E, \eps, \gamma, \alpha)$.

\item If $r_2 \in [r_1, r_3]$ is chosen so that $G \leq 0$ on $[r_1, r_2]$ and $G \geq 0$ on $[r_2, r_3]$, then for any $r_2 \leq s \leq r \leq r_3$ we have
\begin{equation}\label{eqn:main-concl3}
\int_s^r ||\del_ru_r||_{L^2(\del B_1)} dr \leq \left\{ \begin{array}{l l} c ( G(r_3)^{-\gamma} + \delta \gamma \log(r_3/r))^{\frac{\gamma - 1}{2\gamma}}  & \gamma > 0 \\ c G(r_3)^{\frac{1}{2}} (r/r_3)^{\frac{\delta}{2}} & \gamma = 0 \end{array} \right. ,
\end{equation}
while for any $r_1 \leq s \leq r \leq r_2$ we have
\begin{equation}\label{eqn:main-concl4}
\int_s^r ||\del_r u_r||_{L^2(\del B_1)} dr \leq \left\{ \begin{array}{l l} c ( (-G(r_1))^{-\gamma} + \delta \gamma \log(s/r_1))^{\frac{\gamma-1}{2\gamma}}  & \gamma > 0  \\ c  (-G(r_1))^{\frac{1}{2}} (r_1/s)^{\frac{\delta}{2}} & \gamma = 0 \end{array} \right. ,
\end{equation}
where $c = c(c_E, \eps, \gamma, \alpha)$ .
\end{enumerate}
\end{theorem}

\begin{remark}\label{rem:diff-ineq}
It is useful to note the elementary inequality:
\begin{align}\label{eqn:diff-ineq}
||u_r - u_s||_{L^2(\del B_1)} = \left( \int_{\del B_1} \left| \int_s^r \del_r u_r dr \right|^2 d\theta \right)^{1/2} \leq \int_s^r ||\del_r u_r||_{L^2(\del B_1)} dr.
\end{align}
\end{remark}

\begin{remark}\label{rem:mono-only}
Monotonicity \eqref{eqn:main-hyp1} by itself only controls the \emph{square} of the $L^2$ norm:
\[
\int_{r_1}^{r_3} ||r \del_r u_r||_{L^2(\del B_1)}^2 \frac{dr}{r} \leq c_E^{-1} (G(r_3) - G(r_1)),
\]
which is insufficient to control the behavior of $u$ across many scales.
\end{remark}

\begin{remark}
If one allows the energy $\cE(u_r)$ to take values outside $[-1, 1]$, then in place of \eqref{eqn:main-concl1} one has the ODE
\begin{equation}
G' \geq (\delta/r) \min \{ |G|^{1+\gamma}, |G|\} \quad a.e. r \in (r_1, r_3) .
\end{equation}
\end{remark}

Combining Theorem \ref{thm:main} and Remark \ref{rem:diff-ineq}, we obtain directly the following.
\begin{cor}[Uniqueness at $0$ and $\infty$]\label{cor:unique}
Under the same hypotheses as Theorem \ref{thm:main}, if $r_1 = 0$ (and so, necessarily $\Lambda_- = 0$), then there is a limit $u_0 \in L^2(\del B_1)$ so that
\begin{equation}
||u_r - u_0||_{L^2(\del B_1)} \leq \left\{\begin{array}{l l} c \log(r_3/r)^{\frac{\gamma-1}{2\gamma}} & \gamma > 0 \\ c (r/r_3)^{\frac{\delta}{2}} & \gamma = 0 \end{array} \right.  \quad \forall r < r_3 .
\end{equation}
where $c, \delta > 0$ both depend on $(c_E, \eps, \gamma, \alpha)$.

On the other hand if $r_3 = \infty$ (and so $\Lambda_+ = 0$), then there is a limit $u_\infty \in L^2(\del B_1)$ so that
\begin{equation}
||u_r - u_\infty||_{L^2(\del B_1)} \leq \left\{ \begin{array}{l l}c \log(r/r_1)^{\frac{\gamma-1}{2\gamma}} & \gamma > 0 \\ c (r_1/r)^{\frac{\delta}{2}} & \gamma = 0 \end{array} \right. \quad \forall r > r_1.
\end{equation}

If both $r_1 = 0$ and $r_3 = \infty$, then $\del_r u_r \equiv 0$.
\end{cor}

\begin{proof}[Proof of Corollary \ref{cor:unique}]
To prove the first statement, by Theorem \ref{thm:main} it suffices to show that $G(r) \geq 0$.  Otherwise, if we had $G(r_2) < 0$, then for all $0 < r < r_2$ we would have either
\[
-G(r_2) \leq ( (-G(r))^{-\gamma} + \delta \gamma \log(r_2/r))^{-1/\gamma} \leq c(\delta, \gamma) \log(r_2/r)^{-1/\gamma}
\]
or (recalling that $|\cE(u_r)| \leq 1$)
\[
-G(r_2) \leq (r/r_2)^\delta (-G(r)) \leq (1+\Gamma r_2^\alpha) (r/r_2)^\delta,
\]
which is a contradiction for sufficiently small $r$.  The second statement follows by a similar argument.  The third statement follows from the first two, and \eqref{eqn:main-hyp1}.
\end{proof}

\begin{proof}[Proof of Theorem \ref{thm:main}]
For shorthand we write $E(r) = \cE(u_r)$, $F(r) = \cE(z_r)$, $\lambda(r) = \Lambda_+ r^{\alpha} + \Lambda_- r^{-\alpha}$, and we remark that \eqref{eqn:main-hyp2} implies
\begin{equation}
E  - \lambda  \leq F, \quad E' \geq (c_E/r)(F - E) - \lambda/r.
\end{equation}
We first claim there is a $\delta'(c_E, \eps) > 0$ so that
\begin{equation}
E' \geq (\delta'/r)  |E - \lambda |^{1+\gamma} - (2/r)\lambda
\end{equation}
for a.e. $r \in (r_1, r_3)$.  To prove this, we break into four cases.
\begin{enumerate}
\item $F \geq 0$ and $E - \lambda \geq 0$: Then
\begin{align*}
E' \geq (c_E/r)(F - E) - \lambda/r &\geq (c_E/r)(\eps F^{1+\gamma} - \lambda) - \lambda/r \\
&\geq (c_E \eps/r) |E - \lambda|^{1+\gamma} - (2/r) \lambda.
\end{align*}

\item $F \geq 0$ and $E - \Lambda r^\alpha \leq 0$: Then
\begin{align*}
E' \geq (c_E/r)(-E) - \lambda &\geq (c_E/r)|E -  \lambda| - (2/r)\lambda \\
&\geq (c_E/3r) |E - \lambda|^{1+\gamma} - (2/r)\lambda.
\end{align*}

\item $F \leq 0$ and $|F| \leq |E - \lambda|/2$: Since $\lambda \leq 0$, we have:
\begin{align*}
E' \geq (c_E/r)(F - (E - \lambda)) - (2/r)\lambda &\geq (c_E/2r)|E - \lambda| -  (2/r)\lambda \\
&\geq (c_E/6r)|E - \lambda|^{1+\gamma} - (2/r)\lambda.
\end{align*}

\item $F \leq 0$ and $|F| \geq |E - \lambda/2$: Then
\begin{align*}
E' \geq (c_E\eps/r)|F|^{1+\gamma} - (2/r)\lambda &\geq (c_E\eps/4r)|E - \lambda|^{1+\gamma} - (2/r)\lambda.
\end{align*}
\end{enumerate}
This proves our claim.

Fix $\delta'' = \min \{ \delta', 2^{-\gamma} (1+3/\alpha)^{-1-\gamma} \}$.  Recalling that $\max\{\Lambda_+ r_3^\alpha, \Lambda_- r_1^{-\alpha}\} \leq 1$, $\gamma \in [0, 1)$, and $\Gamma_\pm = 3\Lambda_\pm/\alpha$, we compute:
\begin{align*}
G'  &\geq (\delta''/r) |G - (1+3/\alpha)\Lambda_+ r^\alpha - (1-3/\alpha) \Lambda_- r^{-\alpha}|^{1+\gamma}  + 2\Lambda_+ r^{\alpha-1} + 2\Lambda_- r^{-\alpha-1} \\
&\geq (\delta'' 2^{-\gamma}/r)|G|^{1+\gamma} - (\Lambda_+ r^\alpha)^{1+\gamma}/r - (\Lambda_- r^\alpha)^{1+\gamma}/r + 2\Lambda_+ r^{\alpha -1} + 2\Lambda_- r^{-\alpha-1} \\
&\geq (\delta'' 2^{-\gamma}/r)|G|^{1+\gamma} .
\end{align*}
In the second line we used general inequality $|a - b|^{1+\gamma} \geq 2^{-\gamma} |a|^{1+\gamma} - |b|^{1+\gamma}$.  This proves \eqref{eqn:main-concl1}, with $\delta = \delta''/2$.

We work towards proving \eqref{eqn:main-concl2}.  For shorthand write $||v|| = ||v||_{L^2(\del B_1)}$.  We first compute, noting that $G' = E' + (3/r) \lambda$, 
\begin{align}
\left( \int_s^r ||\del_r u_r|| dr \right)^2
&\leq \log(r/s) \int_s^r \int_{\del B_1} r |\del_r u_r|^2 d\theta dr \nonumber \\
&\leq c_E^{-1} \log(r/s) \int_s^r (E' + (2/r) \lambda) dr \nonumber \\
&\leq c_E^{-1} \log(r/s)(G(r) - G(s)) \label{eqn:main1}
\end{align}

Assume for the moment that $\gamma > 0$.  For a.e. $r \in (r_2, r_3)$ we have $G' \geq (\delta/r) G^{1+\gamma}$.  Therefore, for every $r_2 \leq s \leq r < r_3$ we have
\begin{equation*}
(-1/\gamma) ( G(r)^{-\gamma} - G(s)^{-\gamma}) \geq \delta \log(r/s),
\end{equation*}
and hence
\begin{equation}\label{eqn:main2}
G(s) \leq (G(r)^{-\gamma} + \delta \gamma \log(r/s))^{-1/\gamma} .
\end{equation}
Combining \eqref{eqn:main1}, \eqref{eqn:main2}, we get
\begin{align*}
\int_s^r ||\del_r u_r|| dr 
&\leq c_E^{-1/2} \log(r/s)^{1/2} ( G(r_3)^{-\gamma} + \delta \gamma \log(r_3/r))^{-1/2\gamma}
\end{align*}

Let $s_i = e^{-e^i} r_3$.  Choose $i \leq i'$ so that $s_{i'} \geq r \geq s_{i'+1}$ and $s_{i} \geq s \geq s_{i+1}$.  Then we estimate
\begin{align}
\int_s^r ||\del_r u_r|| dr 
&\leq \sum_{j=i'}^i \int_{\max\{s_{j+1}, r_2\}}^{s_j} ||\del_r u_r|| dr \nonumber \\
&\leq c \sum_{j=i'}^i \log(s_j/s_{j+1})^{1/2} (G(r_3)^{-\gamma} + \delta\gamma \log(r_3/s_j))^{-1/2\gamma} \nonumber \\
&\leq c \sum_{j=i'}^i e^{j/2} ( G(r_3)^{-\gamma} + \delta \gamma e^j)^{-1/2\gamma} \nonumber \\
&\leq c (G(r_3)^{-\gamma} + \delta \gamma e^{i'})^{(\gamma-1)/2\gamma} \nonumber \\
&\leq c(c_E, \delta, \gamma)(G(r_3)^{-\gamma} + \delta\gamma \log(r_3/r))^{(\gamma-1)/2\gamma}. \label{eqn:main4}
\end{align}
In the penultimate line we used the inequality
\[
\sum_{j=i'}^\infty e^{j/2} (a+e^j)^{-\beta/2} \leq c(\beta) (a+e^{i'})^{(1-\beta)/2}
\]
for any $\beta > 1$, $a \geq 0$, $i' \geq 0$.

Similarly, if instead $r_1 < s \leq r \leq r_2$, then we have
\[
-G(r) \leq ((-G(s))^{-\gamma} + \delta \gamma \log(r/s))^{-1/\gamma}
\]
and hence
\begin{align*}
\int_s^r ||\del_r u_r|| dr 
&\leq c_E^{-1} \log(r/s)^{1/2}(-G(s))^{1/2} \\
&\leq c_E^{-1} \log(r/s)^{1/2}( (-G(r_1))^{-\gamma} + \delta \gamma \log(s/r_1))^{-1/2\gamma}.
\end{align*}
Therefore, taking now $s_i = e^{e^i} r_1$ we can argue as above to get
\begin{equation}
\int_s^r ||\del_r u_r|| dr \leq c(c_E, \gamma, \delta) ( (-G(r_1))^{-\gamma} + \delta \gamma \log(s/r_1))^{(\gamma-1)/2\gamma}. \label{eqn:main6}
\end{equation}
Together, combining \eqref{eqn:main4}, \eqref{eqn:main6} gives the Dini estimate \eqref{eqn:main-concl2}, \eqref{eqn:main-concl3} in the case when $\gamma > 0$.

Assume now $\gamma = 0$.  Then \eqref{eqn:main-concl1} becomes $G' \geq (\delta/r) |G|$.  If $r_2 \leq s < r < r_3$, then we have
\[
G(s) \leq (s/r)^\delta G(r),
\]
and hence, arguing as before but with $s_i = e^{-i} r_3$, we have
\[
\int_s^r ||\del_r u_r|| dr \leq c(c_E, \delta) G(r_3)^{1/2} (r/r_3)^{\delta/2}.
\]
If instead $r_1 < s < r \leq r_2$, then we have
\[
-G(r) \leq (s/r)^{\delta} (-G(s)),
\]
and (arguing as before but with $s_i = e^i r_1$)
\[
\int_s^r ||\del_r u_r|| dr \leq c(c_E, \delta) (-G(r_1))^{\delta/2} (r_1/s)^{\delta/2} .
\]
This proves \eqref{eqn:main-concl2}, \eqref{eqn:main-concl3} when $\gamma = 0$.
\end{proof}

\section{Minimizers of the Alt-Caffarelli functional}\label{sec:ac}

We consider here the Alt-Caffarelli functional $J$ as defined in Example \ref{ex:ac}.  We shall use the same notation as Example \ref{ex:ac}, and in particular for the duration of this Section we fix a $1$-homogenous minimizer $u_0(x = r\theta) = r \sigma(\theta)$ which is regular away from $0$.

Recall that as outlined in Example \ref{ex:ac} the key point is that if $u \in H^1(B_1)$ is a minimizer for $J$ and satisfies \eqref{eqn:intro-ac-ass1}, \eqref{eqn:intro-ac-ass2}, then for all $r \in (\rho, 1)$ we have the monotonicity
\begin{equation}\label{eqn:ac-mono}
\frac{d}{dr} \cE(u_r) = \frac{n}{r}(\cE(z_r) - \cE(u_r)) + \frac{1}{r} \int_{\del B_1} r^2 |\del_r u_r|^2 
\end{equation}
and (by Theorem \ref{thm:ac-epi}) the symmetric log-epiperimetric inequality
\begin{equation}\label{eqn:ac-epi}
\cE(u_r) \leq \cE(z_r) - \eps |\cE(z_r)|^{1+\gamma} \quad \forall r \in (\rho, 1),
\end{equation}
for $\cE(u) = W(u) - W(u_0)$, $W(u) = W^1_0(u) + |\{ u > 0 \} \cap B_1|$, and $z_r(x) = |x| u_r(x/|x|)$ the $1$-homogenous extension of $u_r$.  \eqref{eqn:ac-mono}, \eqref{eqn:ac-epi} together with Theorem \ref{thm:main} gives a decay-growth estimates \eqref{eqn:intro-ac-grow}, \eqref{eqn:intro-ac-decay} and the Dini estimate
\begin{equation}\label{eqn:ac-dini}
\int_{\rho}^{1} ||\del_r u_r||_{L^2(\del B_1)} dr \leq c( \left[\cE(\rho)_-\right]^{\frac{1-\gamma}{2}} + \left[\cE(1)_+\right]^{\frac{1-\gamma}{2}}) .
\end{equation}

In Subsection \ref{ssec:ac} we prove the symmetric log-epiperimetric inequality (Theorem \ref{thm:ac-epi}).  In the remainder of this Section we highlight and prove several applications.  The first and primary consequence is that regularity and graphicality propogate both outward and inward while the density remains close to constant.
\begin{theorem}[Small density drop implies graphical + estimates]\label{thm:ac-param}
Given $\eta > 0$, there are $\delta(u_0, \eta), \eps(u_0) > 0$, $\gamma(u_0) \in [0, 1)$ so that the following holds.  Let $u \in H^1(B_1)$ minimize $J$, and suppose there is a $1/4 \geq \rho \geq 0$ so that
\begin{equation}\label{eqn:ac-param-hyp1}
\min\{ ||u - u_0||_{L^2(\del B_1)}, ||u_{\rho/2} - u_0||_{L^2(\del B_1)} \} \leq \delta,
\end{equation}
and
\begin{equation}\label{eqn:ac-param-hyp2}
W(u) \leq W(u_0) + \delta, \quad W(u_{\rho/2}) \geq W(u_0) - \delta .
\end{equation}

Then we have
\begin{equation}\label{eqn:ac-param-concl1}
\sup_{r \in [\rho, 1/2]} ||u_r - u_0||_{H^1(\del B_1)} \leq \eta, \quad \sup_{r \in [2\rho, 1/2]} ||u_r -u_0||_{H^1(A_{1, 1/2})} \leq \eta,
\end{equation}
and we can find a $\xi : C\Omega \cap A_{1/2, \rho} \to \R$ so that 
\begin{equation}\label{eqn:ac-param-concl2}
\del \{ u > 0 \} \cap A_{1/2, \rho} = G_{C\Omega}(\xi) \cap A_{1/2, \rho}, \quad \sup_{r \in [2\rho, 1/2]} ||\xi_r||_{C^{2,\alpha}(C\Omega \cap A_{1, 1/2})} \leq \eta.
\end{equation}
Moreover, we have the Dini estimate
\begin{equation}\label{eqn:ac-param-concl3}
\int_{\rho}^{1/2} ||\del_r u_r||_{L^2(\del B_1)} dr \leq c(u_0) \delta^{(1-\gamma)/2}.
\end{equation}
\end{theorem}

\begin{proof}
Assume that $||u - u_0||_{L^2(\del B_1)} \leq \delta$.  The other case is essentially the same.  Let $\rho_*$ be the least radius so that \eqref{eqn:ac-param-concl1}, \eqref{eqn:ac-param-concl2}, \eqref{eqn:ac-param-concl3} hold with $\rho_*$ in place of $\rho$.  By Lemma \ref{lem:ac-extend}, ensuring $\delta(\eta, u_0)$ is sufficiently small, we can assume that $\rho_* \leq \max\{\rho, 1/4\}$.  Moreover, ensuring $\delta(\eta', u_0)$ is small, Lemma \ref{lem:ac-extend} implies we can assume that
\begin{equation}\label{eqn:ac-param-1}
||u_{1/2} - u_0||_{L^2(\del B_1)} \leq \eta'.
\end{equation}

Since there is no loss in generality in assuming $\eta(u_0)$ is small, by \eqref{eqn:ac-mono} and \eqref{eqn:ac-epi} we can apply Theorem \ref{thm:main} with $\cE(v) = W(v) - W(u_0)$ and $\Lambda = 0$ to deduce that
\[
\int_{\rho_*}^{1/2} ||\del_r u_r||_{L^2(\del B_1)} dr \leq c(u_0) \delta^{(1-\gamma)/2}.
\]
From \eqref{eqn:ac-param-1} and \eqref{eqn:diff-ineq} we get
\[
||u_r - u_0||_{L^2(\del B_1)} \leq \eta' + c(u_0) \delta^{(1-\gamma)/2} \quad \forall \rho_* \leq r \leq 1/2.
\]
If $\rho_* > \rho$, then provided $\eta'(\eta, u_0)$, $\delta(u_0, \eta', \eta)$ are small we can use Lemma \ref{lem:ac-extend} to deduce a contradiction.
\end{proof}

\begin{lemma}\label{lem:ac-extend}
Given $\eps > 0$, there is a $\delta(u_0, \eps) > 0$ so that the following holds.  Suppose $u \in H^1(B_1)$ minimizes $J$, and satisfies
\begin{gather}
\min\{||u - u_0||_{L^2(\del B_1)}, ||u_{1/8} - u_0||_{L^2(\del B_1)} \}  \leq \delta, \\
W(u) \leq W(u_0) + \delta, \quad W(u_{1/8}) \geq W(u_0) - \delta.
\end{gather}
Then
\begin{gather}
\sup_{r \in [1/4, 1/2]} ||u_r - u_0||_{H^1(\del B_1)} \leq \eps, \quad ||u - u_0||_{H^1(A_{1/2, 1/4})} \leq \eta ,\\
\del \{ u > 0 \} \cap A_{1/2, 1/4} = G_{C\Omega}(\xi) \cap A_{1/2, 1/4}, \quad ||\xi||_{C^{2,\alpha}(C\Omega \cap A_{1/2, 1/4})} \leq \eps.
\end{gather}
\end{lemma}

\begin{proof}
Straightforward argument by contradiction, similar to \cite[Lemma 4.1]{EnSpVe}.
\end{proof}

\vspace{3mm}

Combining Theorems \ref{thm:ac-param}, \ref{thm:main} gives directly uniqueness at at $0$ and $\infty$.  We state here only the uniqueness at $\infty$; uniqueness at $0$ was first proven in \cite{EnSpVe}.
\begin{cor}[Uniqueness at infinity]\label{cor:ac-unique}
Let $u \in H^1_{loc}(\R^n)$ be an entire minimizer of $J$.  Suppose, for some $r_i \to \infty$, $u_{r_i} \to u_0$ in $L^2_{loc}$. Then there are constants $\gamma(u_0) \in [0, 1)$, $\delta(u_0) > 0$, $C(u)$, so that
\[
||u_r - u_0||_{L^2(\del B_1)} \leq \left\{ \begin{array}{l l}C \log(r)^{\frac{\gamma-1}{2\gamma}} & \gamma > 0 \\ C r^{-\frac{\delta}{2}} & \gamma = 0 \end{array} \right. \quad \forall r > 1.
\]
Here $\gamma$ as in Theorem \ref{thm:ac-epi}.
\end{cor}

A less obvious corollary is the following one-sided perturbation theorem.  \cite{DeJeSh} have classified entire minimizers to $J$ which lie to one side of $u_0$ as fitting inside a Hardt-Simon-type foliation.  Specifically, they have shown
\begin{theorem}[\cite{DeJeSh}]\label{thm:ac-foliation}
There exist regular, entire minimizers $\underline{u} \leq u_0 \leq \overline{u} \in H^1_{loc}(\R^n)$ to $J$, asymptotic to $u_0$ at infinity, with the property that if $u \in H^1_{loc}(\R^n)$ is an entire minimizer to $J$ and $u \leq u_0$ (resp. $u \geq u_0)$, then either $u = \underline{u_t}$ (resp. $u = \overline{u_t}$) for some $t > 0$, or $u = u_0$.
\end{theorem}

By arguing analogously to \cite{Ed}, the following local version of Theorem \ref{thm:ac-foliation} holds.
\begin{theorem}[Local Liouville-type theorem]\label{thm:ac-lious}
Given $\eps > 0$, there is a $\delta(u_0, \eps) > 0$ so that the following holds.  Let $u \in H^1(B_1)$ minimize $J$,
and suppose that
\begin{equation}\label{eqn:ac-lious-hyp}
||u - u_0||_{L^2(B_1)} \leq \delta, \quad u \leq u_0.
\end{equation}

Then either $u = u_0$, or there is a $0 < t \leq \eps$ so that
\begin{equation}\label{eqn:ac-lious-concl}
||u_r - \underline{u}_{rt}||_{L^2(A_{1, 1/2})} \leq \eps \quad \forall 0 < r \leq 1.
\end{equation}
In particular, either $\del \{ u > 0 \} = \del \{ u_0 > 0 \} \cap B_1$, or $\del \{ u > 0 \} \cap B_{1/2}$ is regular, and a small analytic perturbation of $\del \{ \underline{u}_t > 0\} \cap B_{1/2}$.

If one assumes $u \geq u_0$, then the same conclusion holds with $\overline{u}$ in place of $\underline{u}$.
\end{theorem}

\begin{proof}
We show there is a $0 \leq t \leq \eps$ so that \eqref{eqn:ac-lious-concl} holds.  If $t = 0$ then $u = u_0$ follows by the strong maximum principle (\cite{EdSpVe}, or \cite{DeJeSh}).  Suppose, towards a contradiction, this failed: there is a sequence $\delta_i \to 0$, $u_i \in H^1(B_1)$, so that \eqref{eqn:ac-lious-hyp} holds with $u_i$, $\delta_i$ in place of $u, \delta$, but \eqref{eqn:ac-lious-concl} fails every $u_i$ and for every $0 \leq t \leq \eps$.  By standard compactness for minimizers of $J$, we can assume that $u_i \to u_0$ in $H^1_{loc}(B_1) \cap C^0_{loc}(B_1)$.

For $\delta' > 0$ to be determined later, let $\rho_i$ be the least radius so that
\[
|W((u_i)_r) - W(u_0)| \leq \delta' \quad \forall \rho_i < r < 9/10.
\]
By our convergence $u_i \to u_0$, we have $\rho_i \to 0$.  By Theorem \ref{thm:ac-param}, provided we ensure $\delta'(u_0, \eps')$ sufficiently small, we have
\begin{equation}\label{eqn:ac-lious-1}
||(u_i)_r - u_0||_{L^2(A_{1, 1/2})} \leq \eps' \leq \eps \quad \forall 2\rho_i \leq r \leq 1.
\end{equation}
If $\rho_i = 0$, then we obtain a contradiction, so we must have $\rho_i > 0$ for all $i$.

Define $u'_i = u_{\rho_i}$.  Then for every $R > 1$ and $i >> 1$, we have
\begin{equation}\label{eqn:ac-lious-2}
1 = \inf \{ \rho : |W( (u_i')_r) - W(u_0)| \leq \delta' \quad \forall \rho < r < R \}.
\end{equation}
Passing to a subsequence, we can assume that $u_i' \to u'$ in $H^1_{loc}(\R^n) \cap C^0_{loc}(\R^n)$, for some minimizer $u'$ of $J$.  Since every $u_i \leq u_0$, we have $u' \leq u_0$, and therefore by Theorem \ref{thm:ac-foliation} either $u' = u_0$ or $u' = \underline{u}_t$ for some $t > 0$.

If $u' = u_0$, then by our convergence $u'_i \to u'$ and the monotonicity of $r \mapsto W((u'_i)_r)$, for $i >> 1$ we would have
\[
|W( (u'_i)_r) - W(u_0)| \leq \delta'/2 \quad \forall r \in (1/2, 2),
\]
which contradicts \eqref{eqn:ac-lious-2}.  Therefore $u' = \underline{u}_t$ for some $t > 0$.  Since $W((u'_i)_2) \neq 0$, we have $t \leq t_0(u_0)$, and therefore, since $\underline{u}$ is asymptotic to $u_0$ at infinity, there is an $R_0(u_0, \eps')$ so that
\begin{equation}\label{eqn:ac-lious-4}
||\underline{u}_{rt} - u_0||_{L^2(A_{1, 1/2})} \leq \eps' \quad \forall r \geq R_0.
\end{equation}
Combining \eqref{eqn:ac-lious-1}, \eqref{eqn:ac-lious-4}, with our convergence $u'_i \to \underline{u}_t$, we deduce that \eqref{eqn:ac-lious-concl} holds for all $i >> 1$.  This is a contradiction, and finishes the proof of the Theorem.
\end{proof}

\subsection{Proof of symmetric log-epiperimetric}\label{ssec:ac}

We prove in this Section the symmetric log-epiperimetric inequality of Theorem \ref{thm:ac-epi} for Alt-Caffarelli functional \eqref{eqn:ac}.  The proof is in principle a minor modification of the proof of the standard log-epiperimetric of \cite{EnSpVe}.  However, we choose to give the full proof of Theorem \ref{thm:ac-epi} here, partly because there are subtleties involving the choice of constants which are non-obvious even in the original proof, and partly because we can give a more streamlined proof synthesizing both the standard and the ``reverse'' epiperimetric inequalities.

Theorem \ref{thm:ac-epi} is a largely direct consequence of the following two Lemmas, which deal with the ``outer'' variation and ``inner'' variation separately.  The first is verbatim to \cite{SpVe}, and the latter is our modified variant of \cite{EnSpVe}.

\begin{lemma}[{\cite[Lemmas 2.5, 2.6]{SpVe}}]\label{lem:ac-zplus-epi}
Let $\Omega'$ be a fixed, Lipschitz domain in $\del B_1$, and let $\phi_i$ be eigenfunctions of $\Omega'$, with eigenvalues $\lambda_i$ satisfying $\lambda_i - (n-1) \geq \eta > 0$ for every $i \geq 2$ and some $\eta > 0$.  There are numbers $\rho(n, \eta), \eps(n, \eta) \in (0, 1)$ so that the following holds.

Let $z_+ \in H^1(\del B_1)$ take the form $z_+ = \sum_{i =2}^\infty c_i \phi_i$.  Define $h_+ = \sum_{i=2}^\infty c_i r^{\alpha_i} \phi_i$ to be the harmonic extension of $z_+$ to the cone over $\Omega'$, and let $\psi(r)$ be the harmonic function in $A_{1, \rho}$ such that $\psi(r = 1) = 1$, $\psi(r = \rho) = 0$.  Then $\psi h_+ \in H^1(B_1)$ satisfies $\psi h_+|_{\del B_1} = z_+$ and 
\[
W_0(\psi h_+) - W_0(r z_+) \leq -\eps W_0(rz_+) = -\eps|W_0(rz_+)|.
\]
\end{lemma}

\begin{lemma}\label{lem:ac-z1-epi}
There are constants $\delta(\sigma), \eps(\sigma) > 0$, $\gamma(\sigma) \in [0, 1)$, so that the following holds.

Take $\xi \in C^{2,\alpha}(\del \Omega)$, $z_1 \in H^1(\del B_1)$ such that
\begin{equation}\label{eqn:ac-z1-hyp}
||z_1 - \sigma||_{L^2(\del B_1)} \leq \delta, \quad ||\xi||_{2,\alpha} \leq \delta,
\end{equation}
and assume additionally $z_1$ takes the form $z_1 = c_1 \phi^{\Omega_\xi}_1$.  Then we can find an $h_1 \in H^1(B_1)$ satisfying $h_1|_{\del B_1} = z_1$, so that
\begin{equation}\label{eqn:ac-z1-concl}
W(h_1) - W(rz_1) \leq - \eps |W(rz_1) - W(r\sigma)|^{1+\gamma}.
\end{equation}
\end{lemma}

\begin{proof}[Proof of Theorem \ref{thm:ac-epi} given Lemmas \ref{lem:ac-zplus-epi}, \ref{lem:ac-z1-epi}]
Since $\lambda_1^\Omega = n-1$ and $\Omega$ is connected \cite[Theorem 2.3]{EdSpVe}, we have $\lambda_2^\Omega > n-1$.  Therefore, provided $\delta(\sigma)$ is sufficiently small, we have $\lambda_2^{\Omega_\xi} - (n-1) \geq \eps_0(\sigma) > 0$.

Write
\[
z = c_1 \phi_1^{\Omega_\xi} + \sum_{i \geq 2} c_i \phi_i^{\Omega_\xi} =: z_1 + z_+.
\]
By our hypotheses and our previous discussion, we can apply Lemma \ref{lem:ac-zplus-epi} to $z_+$ to obtain an $\psi h_+$, $\rho$, and apply Lemma \ref{lem:ac-z1-epi} to $z_1$ to obtain an $h_1$.

Define $h$ to be the competitor
\[
h(x) = \left\{ \begin{array}{l l} r z_1 + \psi h_+ & \rho \leq |x| \leq 1 \\ \rho h_1(x/\rho) & 0 < |x| \leq \rho \end{array} \right.
\]
By construction, $h \in H^1(B_1)$, and $h|_{\del B_1} = z$ and $h|_{\del B_\rho} = \rho z_1$.  By orthogonality of the $\phi_i$, the facts that $\{ z > 0\} = \{ z_1 > 0\}$ and $\{ h_+ > 0 \} \subset \{ rz_1 > 0 \}$, and the scaling of $W$, it's straightforward to verify that
\begin{align*}
&W(rz) = W(rz_1) + W_0(r z_+), \text{ and}  \\
&W(h) - W(rz) \leq \rho^n(W(h_1) - W(rz_1)) + W_0(\psi h_+) - W_0(rz_+).
\end{align*}

Since by \eqref{eqn:ac-epi-hyp} $W_0(rz_+) \leq c(\sigma)$, we can therefore use Lemmas \ref{lem:ac-zplus-epi}, \ref{lem:ac-z1-epi} to estimate
\begin{align*}
W(h) - W(rz)
&\leq -\rho^n \eps_1 |W(r z_1) - W(r\sigma)|^{1+\gamma} - \eps_2 |W_0(rz_+)| \\
&\leq -\rho^n \eps_1 |W(rz_1) - W(r\sigma)|^{1+\gamma} - (\eps_2/c) |W_0(rz_+)|^{1+\gamma} \\
&\leq -2^{-1-\gamma}\min\{ \rho^n \eps_1, \eps_2/c(\sigma)\} |W(rz) - W(r\sigma)|^{1+\gamma}.
\end{align*}
This proves Theorem \ref{thm:ac-epi}.
\end{proof}

\begin{proof}[Proof of Lemma \ref{lem:ac-z1-epi}]
The proof is similar to \cite{EnSpVe}, except to allow for $\cG$ to be negative we must modify \emph{both} positive and negative modes (i.e. $\xi^\perp_\pm$), and be slightly more careful in our treatment of the zero modes (i.e. $\xi^T$).  We shall use heavily the notation from \cite{EnSpVe}.  In this proof $\eps(\sigma)$, $b(\sigma, \eps)$, $\delta(\sigma, \eps)$ are small positive constants $\leq 1$ which we shall choose as we go along, but can a posteriori be fixed.  Letters $c(\sigma)$, $c'(\sigma)$ represent large constants $\geq 1$ which may increase from line to line.

Recall by \cite[Lemma 3.6]{EnSpVe}, if $v(r, \theta) \in H^1([0, 1] \times \del B_1)$, then $r v(r, \theta) \in H^1(B_1)$ satisfies
\begin{align}\label{eqn:lem-ac-1}
W(rv) = \int_0^1 W_S(v(r)) r^{n-1} dr + \int_0^1 \int_{\del B_1} (\del_r v)^2 r^{n+1} d\theta dr,
\end{align}
where $W_S : H^1(\del B_1) \to \R$ is defined by
\[
W_S(w) = \int_{\del B_1} |\nabla w|^2 - (n-1) w^2 d\theta + \haus^{n-1} (\{w > 0\}).
\]

Our competitor $h_1$, like in \cite{EnSpVe}, will take the form $h_1(r, \theta) = r \kappa(r) \phi_1^{\Omega_{g(r)}}$ for some suitable flows $\kappa(r)$, $g(r, \theta)$.  To this end, we recall the functional introduced by \cite{EnSpVe} $\cG : C^{2,\alpha}(\del \Omega) \times \R \to \R$ given by
\[
\cG(\zeta, s) = (\kappa^2 + s^3) ( \lambda_1^{\Omega_\zeta} - (n-1)) + \haus^{n-1}(\Omega_\zeta) - \haus^{n-1}(\Omega).
\]

As shown in \cite[Lemma 2.3]{EnSpVe}, provided $||\zeta||_{2,\alpha} \leq \delta(\sigma)$ is sufficiently small, $\cG$ is well-defined, analytic, satisfies
\[
\cG(0, 0) = 0, \quad \delta \cG(0, 0) = 0, \quad \delta^2 \cG(0, 0)[(\zeta, r), (\eta, s)] = \delta^2 \cG(0, 0)[(\zeta, 0), (\eta, 0)].
\]
The operator $\delta^2 \cG(0, 0)[(\zeta, 0), (\eta, 0)]$ is a self-adjoint bilinear form on $H^{1/2}(\del \Omega)$, with eigenvalues $\bar\lambda_i \to \infty$, and an $L^2(\del \Omega)$-ON basis of smooth eigenfunctions $\zeta_i$.  Writing $K = \mathrm{span} \{ \zeta_i : \bar\lambda_i = 0\}$, then the kernel of $\delta^2 \cG(0, 0)$ is $K \oplus \R$.

Write $P_K$ for the $L^2(\del\Omega)$-orthogonal projection to $K$, and $P_\pm$ for the projection onto $\mathrm{span} \{ \zeta_i : \pm \bar\lambda_i > 0 \}$.  Note that since $\delta^2 \cG(0, 0)$ has finite index and kernel, we have the bounds
\[
||P_K f||_{2,\alpha} + ||P_+ f||_{2,\alpha} + ||P_- f||_{2, \alpha} \leq c(\sigma) ||f||_{2, \alpha}
\]
for any $f \in C^{2,\alpha}(\del \Omega)$.

Let $Y : (K \oplus \R) \cap U \to K^\perp$ be the Lyapunov-Schmidt map, defined in a suitable $(C^{2,\alpha}(\del\Omega) \oplus \R)$-neighborhood $U$ of $(0, 0)$.  $Y$ is analytic, and satisfies $Y(0, 0) = 0$, $\delta Y(0, 0) = 0$.

\cite[(2.8), Section A.5]{EnSpVe} have shown that
\begin{equation}\label{eqn:lem-ac-2}
|\delta^2 \cG(g, s)[(\zeta, 0), (\zeta, 0)] - \delta^2 \cG(0, 0)[(\zeta, 0), (\zeta, 0)]| \leq (\omega(||g||_{2,\alpha}) + c(\sigma) s^3) ||\xi||_{H^{1/2}}^2,
\end{equation}
for some continuous, increasing function $\omega : [0, \infty) \to [0, \infty)$ with $\omega(0) = 0$.  There is no loss in assuming that $\omega(\tau) \geq \tau$.  From \cite[Section A]{EnSpVe} and \cite[Lemma 2.8]{Da}, we have
\begin{equation}\label{eqn:lem-ac-3}
|\delta^2 \cG(0, 0)[(\zeta, 0), (\zeta, 0)| \leq c(\sigma) ||\zeta||_{H^{1/2}}^2,
\end{equation}
while from \cite[(A13), (A14)]{EnSpVe} we have the opposite inequality
\begin{equation}\label{eqn:lem-ac-4}
||\zeta||_{H^{1/2}}^2 \leq \frac{n}{\lambda_2^\Omega - (n-1)}\left( |\delta^2 \cG(0, 0)[(\zeta, 0), (\zeta, 0)]| + \int_{\del\Omega} \zeta^2 H_{\del\Omega} d\theta\right).
\end{equation}

\vspace{3mm}

As is well-known (see \cite[Theorem 2.3]{EdSpVe}), since $r\sigma$ is a $1$-homogenous minimizer of \eqref{eqn:ac}, $\lambda_1^\Omega = n-1$, $\Omega$ is connected, and $\sigma = \kappa \phi_1^{\Omega}$ for $\kappa$ chosen so that $\del_\nu \sigma = -1$.  From \cite[(2.7)]{EnSpVe}, we have
\[
||\phi_1^{\Omega_\zeta} - \phi_1^{\Omega}||_{L^2(\del B_1)} \leq c(\sigma) ||\zeta||_{L^2(\del\Omega)},
\]
and therefore, combined with \eqref{eqn:ac-z1-hyp}, we deduce $|c_1 - \kappa| \leq c(\sigma) \delta$.  Fix $s_0$ so that $c_1^2 = \kappa^2 + s_0^3$, and note that $|s_0| \leq c(\sigma)\delta^{1/3}$.

Expand
\[
\xi = \left[ P_K(\xi) + Y(\xi, s_0) \right] + \left[ P^\perp_K(\xi) - Y(\xi, s_0) \right] =: \xi^T + \xi^\perp,
\]
and write
\[
\xi^\perp = P_+ \xi^\perp + P_- \xi^\perp =: \xi^\perp_+ + \xi^\perp_-.
\]

Write $P_K \xi = \sum_{j=1}^m \mu_0^j \zeta_j$, for $\zeta_j$ spanning $K$, and write $\mu_0 = (\mu_0^1, \ldots, \mu_0^m) \in \R^m$.  Define $G : \R^m \times \R \to \R$ by $G(\mu, s) = \cG(\sum_{j=1}^m \mu^j \zeta_j + Y(\sum_{j=1}^m \mu^j \zeta_j, s))$.  As shown in \cite{EnSpVe}, there is a $\delta'(\sigma) > 0$ so that provided $|(\mu, s)| \leq \delta'$, $G$ is well-defined, analytic, and (therefore) satisfies the \L ojasiewicz-Simon inequality
\[
|G(\mu, s)|^{1-\beta} \leq c(\sigma) |D G(\mu, s)|,
\]
for some $\beta \in (0, 1/2]$.  $\sigma$ is integrable if and only if $\beta = 1/2$ if and only if $G \equiv 0$.  Note also that $D G(\mu, s) = 0$ implies $G(\mu, s) = 0$.

Since $G(0, 0) = 0$, we have
\[
|G(\mu, s)| \leq c(\sigma)|(\mu, s)| .
\]
By construction we have $|\mu_0| \leq c(\sigma) ||\xi||_{L^2(\del B_1)} \leq c(\sigma) \delta$, and $|s_0| \leq c(\sigma) \delta^{1/3}$.  Therefore, by ensuring $\delta(\sigma)$ is sufficiently small, we can assume
\begin{equation}\label{eqn:lem-ac-5}
\max \{ |(\mu_0, s_0)|, |G(\mu_0, s_0)|^{1-\beta} \} \leq c(\sigma) \delta^{1/6} \leq \min\{ \delta'/4, \kappa^2/4 \}.
\end{equation}

If $G(\mu_0, s_0) = 0$ define $(\mu(t), s(t)) \equiv (\mu_0, s_0)$.  Otherwise, let $(\mu(t), s(t))$ solve the ODE
\[
(\mu'(t), s'(t)) = -\frac{D G(\mu, s)}{|D G(\mu, s)|}, \quad (\mu(0), s(0)) = (\mu_0, s_0).
\]
The solutions $(\mu(t), s(t))$ exists smoothly on some positive, maximal time interval $[0, t_*)$, where either $t_* = \infty$, or $\lim_{t \to t_*} |(\mu(t), s(t)| = \delta'$, or $\lim_{t \to t_*} G(\mu(t), s(t)) = 0$.  Since
\[
|(\mu(t), s(t)) - (\mu_0, s_0)| \leq t, \quad |(\mu_0, s_0)| \leq \delta'/4,
\]
we can assume that either $t_* \geq \delta'/2$ or $\lim_{t \to t_*} G(\mu(t), s(t)) = 0$.  Note also that
\[
\frac{d}{dt} G(\mu(t), s(t)) = - |D G(\mu(t), s(t))| < 0,
\]
and so $G(\mu(t), s(t))$ is decreasing.

If $G(\mu_0, s_0) = 0$, set $b = 0$ and $\eta(r) \equiv 0 \equiv b |G(\mu_0, s_0)|^{1-\beta}(1-r)$.

If $G(\mu_0, s_0) < 0$, then $t_* \geq \delta'2$.  Let $\eta(r) = b |G(\mu_0, s_0)|^{1-\beta}(1-r)$.  By our choice of $\delta$ above we have $|\eta(r)| \leq \delta'/2$ for $r \in [0, 1]$, and so $\mu(\eta(r)), s(\mu(r))$ are well-defined on $[0, 1]$.  We have the bounds
\begin{align*}
G(\mu(\eta(r)), s(\eta(r))) - G(\mu_0, s_0)
&\leq -\int_0^{\eta(r)} |D G(\mu(t), s(t))| dt \\
&\leq -(1/c) \int_0^{\eta(r)} |G(\mu(t), s(t))|^{1-\beta} dt \\
&\leq -(1/c) |G(\mu_0, s_0)|^{1-\beta} \eta(r) \\
&\leq -(b/c) |G(\mu_0, s_0)|^{2-2\beta} (1-r) \quad \forall r \in [0, 1].
\end{align*}

If $G(\mu_0, s_0) > 0$, we break into two cases.  Define
\[
t_1 = \sup \{ t \in [0, \delta'/2] : G(\mu(t), s(t)) \geq G(\mu_0, s_0)/2 \}.
\]
If $t_1 \geq b G(\mu_0, s_0)^{1-\beta}$, then define $\eta(r) = b G(\mu_0, s_0)^{1-\beta}(1-r)$ as before, and estimate
\begin{align*}
G(\mu(\eta(r)), s(\eta(r))) - G(\mu_0, s_0) 
&\leq -(1/c) \int_0^{\eta(r)} |G(\mu(t), s(t))|^{1-\beta} dt\\
&\leq -(1/2c) |G(\mu_0, s_0)|^{1-\beta} \eta(r) \\
&\leq -(b/c) |G(\mu_0, s_0)|^{2-2\beta} (1-r) \quad \forall r \in [0, 1].
\end{align*}
If $t_1 < b G(\mu_0, s_0)^{1-\beta}$, then take $\eta(r) : [0, 1] \to \R$ to be a smooth, decreasing function satisfying
\[
\eta|_{[0, 1/2]} = t_1, \quad \eta(1) = 0, \quad |\eta'| \leq 3t_1.
\]
In this case we estimate (recalling that $2-2\beta \geq 1$ and $b \leq 1$)
\begin{align*}
G(\mu(\eta(r)), s(\eta(r))) - G(\mu_0, s_0) 
&= -G(\mu_0, s_0)/2 \\
&\leq -(b/2) |G(\mu_0, s_0)|^{2-2\beta}(1-r) \quad \forall r \in [0, 1/2].
\end{align*}

Observe that, however we defined $\eta$, we have the bounds $|\eta'(r)| \leq 3b |G(\mu_0, s_0)|^{1-\beta}$ and
\begin{align}
&G(\mu(\eta(r)), s(\eta(r))) - G(\mu_0, s_0) \leq 0 \quad \forall r \in [0, 1], \\
&G(\mu(\eta(r)), s(\eta(r))) - G(\mu_0, s_0) \leq -(b/c)|G(\mu_0, s_0)|^{2-2\beta}(1-r) \quad \forall r \in [0, 1/2]. \label{eqn:lem-ac-6}
\end{align}
Moreover, by \eqref{eqn:lem-ac-5} and our definition of $\eta(r)$ we have the bounds
\begin{equation}\label{eqn:lem-ac-7}
|(\mu(\eta(r)), s(\eta(r)))| \leq c(\sigma) \delta^{1/6} \quad \forall r \in [0, 1].
\end{equation}

Define $\eta_-(r) = 1 + a(1-r)\eps$.  Then
\begin{align*}
\int_0^1 (\eta_-(r)^2 - 1)r^{n-1} dr = \frac{2a \eps}{n(n+1)} \pm c(n) a^2 \eps^2 \geq 4\eps/n
\end{align*}
provided $a(n)$ is chosen sufficiently large, and $\eps(n)$ sufficiently small.  Fix $a$ to be thus.  Similarly, define $\eta_+(r) = 1 - a'(1-r)\eps$, and then
\begin{align*}
\int_0^1 (\eta_+(r)^2 - 1)r^{n-1} dr = \frac{-2a'\eps}{n(n+1)} \pm c(n) a'^2 \eps^2 \leq -4\eps/n .
\end{align*}
provided we choose and fix $a'(n)$ large, and ensure $\eps(n)$ is small.

\vspace{3mm}

We define our competitor as follows.  First we define
\begin{align*}
g(r, \theta) &= \left[ \sum_j \mu^j(\eta(r)) \zeta_j + Y(\sum_j \mu^j(\eta(r)) \zeta_j, s(\eta(r)))\right] + \left[\eta_-(r) \xi^\perp_- + \eta_+(r) \xi^\perp_+\right] \\
&=: g^T(r, \theta) + g^\perp(r, \theta),
\end{align*}
and then set
\[
h_1(x = r\theta) = r \kappa(r) \phi_1^{\Omega_{g(r)}}(\theta), \quad \text{ for } \quad \kappa(r)^2 = \kappa^2 + s(\eta(r))^3.
\]

From \eqref{eqn:lem-ac-1} and the form of $h_1$, we have
\begin{align}
W(h_1) - W(rz)
&= \int_0^1  (\cG(g(r), s(\eta(r))) - \cG(\xi, s_0))r^{n-1} dr \\
&\quad + \int_0^1 \int_{\del B_1} (\del_r (h_1(r)/r))^2 r^{n+1} d\theta dr . \label{eqn:lem-ac-8}
\end{align}

Provided $\delta(\sigma)$ is sufficiently small, we have
\begin{align*}
&(\del_r g)^2 \leq c(\sigma) |\eta'(r)|^2 + c(\sigma) \eps^2( |\xi^\perp_-(r, \theta)|^2 + |\xi^\perp_+(r, \theta)|^2), \\
&\kappa'(r)^2 \leq c(\sigma) |\eta'(r)|^2
\end{align*}
and therefore, using \cite[(2.7)]{EnSpVe}, we deduce
\begin{align}
\int_{\del B_1} (\del_r (h_1(r)/r))^2  d\theta
&\leq 2 \kappa'(r)^2 + 2 \kappa(r)^2 \int_{\del B_1} (\delta \phi_1^{g(r)}[\del_rg(r)])^2 d\theta \nonumber \\
&\leq c(\sigma) |\eta'(r)|^2  + c(\sigma) \kappa(r)^2 ||\del_r g(r)||^2_{L^2(\del B_1)} \nonumber \\
&\leq c(\sigma) b^2 |G(\mu_0, s_0)|^{2-2\beta} + c(\sigma) \eps^2 ||\xi^\perp||^2_{L^2(\del B_1)}. \label{eqn:lem-ac-9}
\end{align}

We work towards estimating the first term in \eqref{eqn:lem-ac-8}.  Write
\begin{align*}
&\cG(g(r), s(\eta(r))) - \cG(\xi, s_0)\\
&= \left[ (\cG(g(r), s(\eta(r))) - G(\mu(\eta(r)), s(\eta(r)))) - (\cG(\xi, s_0) - G(\mu_0, s_0) \right] \\
&\quad + \left[ G(\mu(\eta(r)), s(\eta(r))) - G(1) \right] \\
&=: E^\perp(r) + E^T(r).
\end{align*}
Expand $\xi^\perp = \sum_i \alpha_i \zeta_i \equiv \sum_{\bar\lambda_i \neq 0} \alpha_i \zeta_i$.

By considering the Taylor expansion of
\[
t \mapsto \cG( g^T(r) + t g^\perp(r), s(\eta(r))),
\]
we deduce there is a $\tau \in (0, 1)$ so that, provided $\delta(\sigma)$ is sufficiently small:
\begin{align*}
&\cG(g(r), s(\eta(r))) - G(\mu(\eta(r)), s(\eta(r))) \\
&= \delta \cG(g^T(r))[(g^\perp(r), 0)] + \frac{1}{2} \delta^2 \cG(g^T(r) + \tau g^\perp(r))[(g^\perp(r), 0), (g^\perp(r), 0)] \\
&= \frac{1}{2} \delta^2 \cG(g^T(r) + \tau g^\perp(r))[(g^\perp(r), 0), (g^\perp(r), 0)] \\
&= \frac{1}{2} \delta^2 \cG(0)[(g^\perp(r), 0), (g^\perp(r), 0)] \pm (\omega( 2 ||g(r)||_{2,\alpha}) + c(\sigma) s(\eta(r))^3)||g^\perp(r)||^2_{H^{1/2}} \\
&= \eta_-(r)^2/2 \sum_{\bar\lambda_i < 0} \bar\lambda_i \alpha_i^2 + \eta_+(r)^2/2 \sum_{\bar \lambda_i > 0} \bar\lambda_i \alpha_i^2 \pm ( \omega(c \delta^{1/6}) + c(\sigma) \delta^{1/2}) ||\xi^\perp||_{H^{1/2}}^2 \\
&= - \eta_-(r)^2/2 \sum_{\bar\lambda_i < 0} |\bar\lambda_i| \alpha_i^2 + \eta_+(r)^2/2\sum_{\bar\lambda_i > 0} |\bar\lambda_i| \alpha_i^2 \pm 2 \omega( \delta^{1/7}) ||\xi^\perp||_{H^{1/2}}^2.
\end{align*}
In the third line we used \cite[(B4)]{EnSpVe}, the fourth we used \cite[(2.8), Section A5]{EnSpVe}, and in the fifth line we used \eqref{eqn:lem-ac-7} to estimate
\[
||g(r)||_{2,\alpha} \leq c(\sigma) |\mu(\eta(r))| + c(\sigma)||\xi||_{2,\alpha} \leq c(\sigma) \delta^{1/6}.
\]

Similarly, again taking $\delta(\sigma)$ is small we can estimate
\begin{align}
\cG(\xi, s_0) - \cG(\mu_0, s_0) 
&= -(1/2) \sum_{\bar\lambda_i < 0} |\bar\lambda_i| \alpha_i^2 + (1/2) \sum_{\bar\lambda_i > 0} |\bar\lambda_i| \alpha_i^2 \pm 2 \omega(\delta^{1/7})||\xi^\perp||_{H^{1/2}}^2. \label{eqn:lem-ac-11}
\end{align}
Recalling our choice of $\eta_-(r)$, $\eta_+(r)$, we deduce
\begin{align}
\int_0^1 E^\perp(r) r^{n-1} dr
&\leq -(2\eps/n) \sum_{\bar\lambda_i \neq 0} |\bar\lambda_i| \alpha_i^2 + c(\sigma) \omega(\delta^{1/7}) ||\xi^\perp||_{H^{1/2}}^2. \label{eqn:lem-ac-10}
\end{align}


Combining \eqref{eqn:lem-ac-6}, \eqref{eqn:lem-ac-9}, \eqref{eqn:lem-ac-10}, and ensuring $b(\sigma)$ is sufficiently small, we deduce
\begin{align}
W(h_1) - W(rz) 
&\leq -(b/c) |G(\mu_0, s_0)|^{2-2\beta} + c b^2 |G(\mu_0, s_0)|^{2-2\beta} \nonumber \\
&\quad - (2\eps/n) \sum_{\bar\lambda_i \neq 0} |\bar \lambda_i| \alpha_i^2 + (c \omega(\delta^{1/7}) + c \eps^2) ||\xi^\perp||^2_{H^{1/2}} \nonumber \\
&\leq -(b/2c) |G(\mu_0, s_0)|^{2-2\beta}  - (\eps/n) \sum_i |\bar\lambda_i| \alpha_i^2  \nonumber \\
&\quad + (c \omega(\delta^{1/7}) + c \eps^2 - \eps/c)||\xi^\perp||_{H^{1/2}}^2,\label{eqn:lem-ac-12}
\end{align}
for $c = c(\sigma)$.  In the second inequality we additionally used \eqref{eqn:lem-ac-4} to estimate
\begin{align*}
\sum_i |\bar \lambda_i| \alpha_i^2
&\geq \frac{1 + \min_{\bar\lambda_i \neq 0} |\bar\lambda_i|}{2} ( |\delta^2 \cG(0, 0)[(\xi^\perp, 0), (\xi^\perp, 0)]| + ||\xi^\perp||^2_{L^2}) \\
&\geq (1/c(\sigma))||\xi^\perp||_{H^{1/2}}^2.
\end{align*}

Recall from \eqref{eqn:lem-ac-11} that we have (ensuring $\delta(\sigma)$ is small):
\begin{align}
n(W(rz) - W(r\sigma))
&= \cG(\xi, s_0) - \cG(\mu_0, s_0) + \cG(\mu_0, s_0) \nonumber \\
&= (1/2) \sum_i \bar\lambda_i \alpha_i^2 + G(\mu_0, s_0) \pm 2\omega(\delta^{1/7})||\xi^\perp||_{H^{1/2}}^2. \label{eqn:lem-ac-13}
\end{align}

If $G(\mu_0, s_0) = 0$, then we can combine \eqref{eqn:lem-ac-12}, \eqref{eqn:lem-ac-13}, to estimate
\begin{align}
W(h_1) - W(rz) 
&\leq -\eps|W(rz) - W(r\sigma)| + (c \omega(\delta^{1/7}) + c\eps^2 - \eps/c) ||\xi^\perp||^2_{H^{1/2}} \\
&\leq -\eps|W(rz) - W(r\sigma)|
\end{align}
provided $\eps(\sigma)$ is small, and $\delta(\sigma, \eps)$ is small.

Suppose $G(\mu_0, s_0) \neq 0$.  Using \eqref{eqn:lem-ac-3}, \eqref{eqn:lem-ac-13}, and the inequality $|a - b|^{1+\gamma} \geq 2^{-\gamma} |a|^{1+\gamma} - |b|^{1+\gamma}$ for any $a, b \in \R$, we get
\begin{align}
|G(\mu_0, s_0)/n|^{1+\gamma}
&\geq 2^{-\gamma} |W(rz) - W(r\sigma)|^{1+\gamma} - \left| (1/2n) \sum_i \bar\lambda_i \alpha_i^2 \pm 2\omega(\delta^{1/7}) ||\xi^\perp||_{H^{1/2}}^2\right|^{1+\gamma} \\
&\geq 2^{-\gamma} |W(rz) - W(r\sigma)|^{1+\gamma} - c(\sigma) ||\xi^\perp||_{H^{1/2}}^2
\end{align}
and therefore, writing $c' = c'(\sigma)$, 
\begin{align}
W(h_1) - W(rz)
&\leq (-b/c') |W(rz) - W(r\sigma)|^{1+\gamma} \\
&\quad + ( c' \omega(\delta^{1/7}) + c' b + c'\eps^2 - \eps/c') ||\xi^\perp||_{H^{1/2}}^2 \\
&\leq (-b/c') |W(rz) - W(r\sigma)|^{1+\gamma}
\end{align}
provided $\eps(\sigma)$ is small, and $\delta(\sigma, \eps)$ is small, and $b(\sigma, \eps)$ is small.
\end{proof}

\section{Obstacle and thin-obstacle problem}\label{sec:fb}

Similar results to the ones presented in the previous sections hold for almost minimizers of the obstacle problem and of the thin-obstacle problem. Since the modification for the thin-obstacle problem are essentially already contained in \cite{cospve}, we present here only the statements and proofs for almost minimizers of the obstacle problem. 

\subsection{Results for almost minimizers of the obstacle problem}
In this section we follow the notations and arguments of \cite{cospve}. Let $B_1$ be the unit ball in $\R^n$ and consider the functional 
$$
\cF_{\text{\tiny\sc ob}}(u,W):=\frac12\int_{W}|\nabla u|^2\,dx+\int_{W}u\,dx\,,
$$
where we will drop the dependence on the set if $W=B_1$, and the set of admissible functions 
$$\mathcal K_{\text{\tiny\sc ob}}:=\big\{u\in H^1(B_1)\ :\  u\ge 0\ \text{in}\ B_1\big\}.$$

\begin{definition}
A function $u\in \mathcal K_{\text{\tiny\sc ob}}$ is a \emph{$(\Lambda, \alpha, r_0)$ almost minimizer of the obstacle problem in $B_1$} if 
\begin{equation}\label{e:intro:ostacolo}
\mathcal \cF_{\text{\tiny\sc ob}}(u, W)\leq (1+\Lambda \,r^\alpha)\,\mathcal \cF_{\text{\tiny\sc ob}}(v,W) \qquad \forall v\in\mathcal K_{\text{\tiny\sc ob}}
\end{equation}
for all $W \subset\subset B_r(x) \subset U$ with $0 < r < r_0$. 
\end{definition}

The relevant energy is given by the Weiss' boundary adjusted energy
$$
\mathcal E(u)=W_0^2(u) + \int_{B_1} \max\{ u, 0\}-\int_{B_1}\max\{\sigma,0\}\,, 
$$
where $\sigma$ is a $2$-homogeneous global minimizer of the energy. In particular we will assume $\sigma$ to belong to the following class:
\[
\mathcal B := \{Q_A \colon \R^d \to \R \,:\, Q_A(x) = x \cdot Ax,\,\text{ A symmetric non-negative with }\tr A = 1/4\} \,.
\]
It is a simple conputation to see that $\mathcal E$ satisfies the almost monotonicity \ref{ass_symepi} (1) when $u$ is rescaled as $u_{r}(x)=u(rx)/r^2$ (see for instance \cite{cospve}). So in order to apply Theorem \ref{thm:main} we only need to check the symmetric epiperimetric inequality. This follows as a minor modification of \cite[Proposition 3.1]{cospve}, which we outline in the next section for the reader's convenience.
\begin{prop}[Symmetric epiperimetric inequality for the obstacle problem]\label{p:epio}
Let $\sigma \in \mathcal B$. There are constants $\delta(\sigma),E(\sigma) > 0$, $\gamma(\sigma) \in [0, 1)$, so that the following holds.

Let $z\in H^1(\partial B_1)\cap\mathcal K$ be such that
	$$\|z-\eta\|_{L^2(\partial B_1)}\leq \delta\qquad\text{and}\qquad |\mathcal F_{\text{\tiny\sc ob}}(r z)-\mathcal F_{\text{\tiny\sc ob}} (\sigma)|\le E\,.$$ 
Then there is an $h \in H^1(B_1)$ with $h|_{\del B_1} = z$, so that
\begin{equation}
\mathcal E(h) - \mathcal E(r^2 z)  \leq - \eps |\mathcal E(r^2 z)|^{1+\gamma}.
\end{equation}

\end{prop}

Using Proposition \ref{p:epio} and the monotonicity of $\mathcal E$, we can apply once again Theorem \ref{thm:main} to the energy $\mathcal E(u_r)=W_2(u_r)$ to obtain the following uniqueness of blow-ups and blow-down at singular points. We remark that the uniqueness for blow ups with logarithmic decay had already been established in \cite{cospve1}.

\begin{cor}[Uniqueness of blow-downs for almost minimizers of obstacle problem]\label{cor:o-unique}
Let $u \in H^1_{loc}(\R^n)$ be an entire minimizer of $\mathcal F_{\text{\tiny\sc ob}}$.  Suppose, for some $r_i \to \infty$, $u_{r_i} \to u_0$ in $L^2_{loc}$. Then there are constants $\gamma(u_0) \in [0, 1)$, $\delta(u_0) > 0$, $C(u)$, so that
\[
||u_r - u_0||_{L^2(\del B_1)} \leq \left\{ \begin{array}{l l}C \log(r)^{\frac{\gamma-1}{2\gamma}} & \gamma > 0 \\ C r^{-\frac{\delta}{2}} & \gamma = 0 \end{array} \right. \quad \forall r > 1.
\]
Here $\gamma$ as in Proposition \ref{p:epio}
\end{cor}

\subsection{Proof of the symmetric epiperimetric inequality Proposition \ref{p:epiK}} The symmetric epiperimetric inequality of Propostion \ref{p:epio} follows from the following  general proposition, analogous to \cite[Proposition 3.1]{cospve}. We fix $\mathcal H=L^2(\partial B_1)$, $\mathcal K\subset \mathcal H$ a convex cone, $\mathcal W=H^2(\partial B_1)$.

\begin{prop}\label{p:epiK}
Let $\mathcal G$ be a functional satisfying assumptions (SL) and (FL) of Proposition 3.1 in \cite{cospve}, while assumption (\L S) is replaced by

\smallskip

(\L S') $\cF$ has the constrained \L ojasiewicz property, that is there are constants $\gamma\in\,]0,1/2]$, $C_L>0$, $\delta_L>0$ and $E_L>0$, depending on $\psi\in \mathcal W\cap\mathcal K$ a critical point of $\cF$ in $\mathcal K$, such that for every critical point $\mathcal \varphi \in \mathcal S$ the following inequality holds 
\begin{equation}\label{e:main:loja}
\big|\mathcal F(u)-\mathcal F(\psi)\big|^{1-\gamma}\le C_L\|\nabla\mathcal F(u)\|_{\mathcal K},
\end{equation}
for every $u\in \mathcal K\cap\mathcal W$ such that $\|u-\varphi\|\le \delta_L$ and $|\mathcal F(u)-\mathcal F(\varphi)|\le E_L$.

\smallskip

Then there are constants $\delta_0>0$ and $E>0$, depending only on the dimension and $\psi$, such that: if $c\in H^1(\partial B_1)\cap\mathcal K$ satisfies  
	$$\|c-\psi\|_{L^2(\partial B_1)}\leq \delta_0\qquad\text{and}\qquad |\mathcal F(c)-\mathcal F(\psi)|\le E,$$ 
	then there exists a function $h=h(r,\theta)\in H^1(B_1)$ satisfying $h(r,\cdot)\in\mathcal K$, for every $r\in(0,1]$, and 
	\begin{equation}\label{e:logepiK}
	\mathcal G(h)-\mathcal G(\phi)\leq \left(\mathcal G(z)-\mathcal G(\phi) \right) -\eps \left|\mathcal G(z)-\mathcal G(\phi)\right|^{2-2\gamma} 
	\end{equation}
	where $\phi(r,\theta):=r^k\psi(\theta)$, $z(r,\theta):=r^k c(\theta)$, $\eps>0$ is a universal constant and $\gamma>0$ is the exponent from  (\L S). 
\end{prop}

\begin{proof}
The proof is the same as that of \cite[Proposition 3.1]{cospve}, where one replaces the choice of $\eps_2$ in equation (3.5) with 
	\begin{equation}\label{e:eps_2}
	|\cF(u_0)-\cF(\psi)|\leq 2 \big|\cF(u(t))-\cF(\psi)\big)|\quad\text{for every} \quad 0< t \leq \eps_2\,,
	\end{equation}
and in the last string of inequalities in page 16 one uses \eqref{e:main:loja} to replace $\cF(u(t))-\cF(\psi)$ with $|\cF(u(t))-\cF(\psi)|$.
\end{proof}

\begin{proof}[Proof of Proposition {\ref{p:epio}}] The proof follows if we can verify the assumptions of Proposition \ref{p:epiK}. To this aim, we observe that properties $(SL)$ and $(FL)$ hold when we set
$$
\mathcal G: =\cF_{\text{\tiny\sc ob}} \qquad \text{and}\qquad \cF(\phi):= \int_{\partial B_1}\left(|\nabla_\theta \phi|^2-2d\,\phi^2\right)\,d\cH^{n-1}\,,
$$
as explained in the proof of \cite[Theorem 1.10]{cospve}. So we only need to show that property $(\L S')$ holds for $\cF$. Even though this property is stronger than property $(\L S)$, its proof is similar to that of \cite[Proposition 4.4]{cospve}. So following the reasoning there we let $\varphi\in\mathcal S$ be such that 
$$\|u-\varphi\|_2\le 2\delta.$$ 
Notice that $u-\varphi$ can be uniquely decomposed in Fourier series as $u-\varphi=Q_-+Q_0+\eta$, where $Q_-$ contains only lower eigenmodes (corresponding to eigenvalues $<2d$), $Q_0$ is a $2d$-eigenfunction and $$\displaystyle{\eta(x)=\sum_{\{j:\lambda_j>2d\}}c_j\phi_j(x)},$$ 
which contains only higher eigenmodes (corresponding to eigenvalues $>2d$).  Thus, 
$$u=Q_-+Q_0+\varphi+\eta\quad \text{ and }\quad \|Q_-\|_2, \|Q_0\|_2, \|\eta\|_2\le 2\delta\,.$$
We now consider 
$$M:=\max_{x\in {\mathcal M}}\{-2Q_-(x)-Q_0(x)-\varphi(x)\}$$ 
and suppose that the maximum is realized at a point $x_M\in\mathcal M$. Notice that since $Q_-+Q_0$ is a finite sum of (smooth) eigenfunctions, there is a constant $C>0$ (depending onnly on the dimension) such that $\|Q_-+Q_0\|_{L^\infty}\le C\delta$. Thus, if $M>0$, then $x_M\in\{\varphi<C\delta\}$ and $M\le C\delta$. We now choose $\delta$ such that $10C\delta< c_d:=(2d)^{-\frac12}$  and we claim that the function
	$$\tilde u=2Q_-+Q_0+\varphi+\frac{2M}{c_d}\big(c_d-\varphi\big)$$
	is non-negative. Indeed, it is sufficient to consider the following two cases:
	
	$\bullet$ on the set $\{\varphi\ge 2C\delta\}$, we have that  
	$$\tilde u=\left(2Q_-+Q_0+\frac12\varphi\right)+2M+\varphi\left(\frac12-\frac{2M}{c_d}\right)\ge 0,$$
	since each of the three terms is non-negative;
	
	$\bullet$ on the set $\{\varphi\le 2C\delta\}$, we have that
	$$\tilde u\ge 2Q_-+Q_0+\varphi+\frac{2M}{c_d}\big(c_d-2C\delta\big)\ge   2Q_-+Q_0+\varphi+M\ge 0.$$

Next, using the fact that $c_d-\varphi$ is a $2d$-eigenfunction (notice that the integral of $c_d-\varphi$ on $\partial B_1$ vanishes, due to the fact that $2d>0$), we calculate 
	\begin{equation*}
	\begin{split}
	-(\tilde u-u)\cdot\nabla \cF(u)
	&= \int_{\partial B_1}\big(-\Delta  u - 2d  u + 1\big) \left(-Q_-+\eta-\frac{2M}{c_d}\left(c_d-\varphi\right)\right)\, d\cH^{d-1}\\
	&= \int_{\partial B_1}\big(-\Delta  u - 2d  u + 1\big) \,\left(\eta-Q_-\right)\, d\cH^{d-1}\\
	&= \int_{\partial B_1}\big(-\Delta  (Q_-+Q_0+\eta) - 2d  (Q_-+Q_0+\eta)\big) \,\left(\eta-Q_-\right)\, d\cH^{d-1}\\
	&= \int_{\partial B_1} \left(|\nabla \eta|^2 -\lambda \eta^2\right) \,d\cH^{d-1}-\int_{\partial B_1} \left(|Q_-\eta|^2 -\lambda Q_-^2\right)\, d\cH^{d-1} \\
	&=2\left(\mathcal F(\eta) - \mathcal F(Q_-)\right).
	\end{split}
	\end{equation*}
		Notice that since the set of eigenvalues is discrete, there is a (spectral gap) constant $G(d)>0$ such that $|\lambda_j-2d|\ge G(\lambda)$, whenever $|\lambda_j-2d|>0$. In particular, we have the inequalities 
		\begin{gather}
		2\mathcal F(\eta)= \sum_{j:\lambda_j>2d} c_j^2(\lambda_j-2d)\ge G(\lambda)\sum_{j:\lambda_j>\lambda} c_j^2=G(\lambda)\|\eta\|_2^2\notag\\
		-2\mathcal F(Q_-)= -\sum_{j:\lambda_j<2d} c_j^2(\lambda_j-2d)\ge G(\lambda)\sum_{j:\lambda_j>2d} c_j^2=G(\lambda)\|Q_-\|_2^2\notag
		\end{gather}
		Thus, we get by definition of $\|\cdot\|_{\mathcal K}$,
	\begin{equation}\label{e:ostacolo:nablaF}
	\begin{split}
	\|\nabla \cF(u)\|_{\mathcal K}& \ge \frac{-(\tilde u-u)\cdot\nabla \cF(u)}{\|u-\tilde u\|_{2}} \geq\frac{2(\mathcal F(\eta)-\mathcal F(Q_-))}{\frac{2M}{c_\lambda}\|c_\lambda-\varphi\|_{2}+\|\eta\|_{2}+\|Q_-\|_{2}} \\
	&\geq \frac{2(\mathcal F(\eta)-\mathcal F(Q_-))}{M\frac{2}{c_d}(\|c_d\|_2+\|\varphi\|_{2})+(2G(\lambda)\mathcal F(\eta))^{\frac12}+ (-2G(\lambda)\mathcal F(Q_-))^{\frac12}}\\
	&\geq C\frac{\mathcal F(\eta)-\mathcal F(Q_-)}{M+3\left(\mathcal F(\eta)-\mathcal F(Q_-)\right)^{\frac12}},\end{split}
	\end{equation}
where $C$ is a constant depending only on the dimension.

Now, in order to get \eqref{e:main:loja}, it only remains to estimate $M$ and put together \eqref{e:ostacolo:nablaF} and \eqref{e:ostacolo:F-F}. We notice that: 

$\bullet$ $2Q_-+Q_0+\varphi$ is a (finite) linear combination of (orthonormal and smooth) eigenfunctions corresponding to eigenvalues $\le2d$; 

$\bullet$ the $L^2$ norm of $2Q_-+Q_0+\varphi$ is bounded by a universal constant. 

\noindent As a consequence, there is a universal (Lipschitz) constant $L$, depending only on the dimension, such that $\|\nabla (2Q_-+Q_0+\varphi)\|_{L^\infty(\mathcal M)}\le L$.  
\noindent Thus, since the negative part $\psi:=-\inf\{(2Q_-+Q_0+\varphi),0\}$ is such that $\sup\psi=M$ is small enough (bounded by a dimensional constant, as already mentioned above), we get that there is a constant $C$ (depending on the dimension) such that 
$$\|\psi\|_{L^2(\partial B_1)}^2\ge C L^{-d}M^{d+2}=C L^{-d}\|\psi\|_{L^\infty(\partial B_1)}^{d+2}.$$
Since $u\ge 0$ on $\partial B_1$, we have that $\psi\le \eta-Q_-$ and so, 
$$M^{d+2}\le C^{-1}L^d\left(\|\eta\|_2^2+\|Q_-\|_2^2\right)\le \frac{2L^d}{CG(\lambda)}\left(\mathcal F(\eta)-\mathcal F(Q_-)\right),$$
which, together with \eqref{e:ostacolo:nablaF} and \eqref{e:ostacolo:F-F}, implies that, if we set $\gamma=\frac{1}{d+2}$, then
$$
\|\nabla \cF(u)\|_{\mathcal K}\geq C\, \left(\mathcal F(\eta)-\mathcal F(Q_-)\right)^{1-\gamma}\,.
$$
Using that, by orthogonality, it holds	
	\begin{align}\label{e:ostacolo:F-F}
	|\cF(u)- \cF(\varphi)| =|\mathcal F(\eta)+\mathcal F(Q_-)|\leq \mathcal F(\eta)-\mathcal F(Q_-)\,,
	\end{align}
we conclude \eqref{e:main:loja}. 

\end{proof}

\section{Almost-minimizing currents}\label{sec:am}

We are interested here in the regularity properties of almost minimizing currents.  Given $\alpha > 0$, $\Lambda \geq 0$, and $r_0 > 0$, an integral $n$-current $T$ in an open subset $U \subset \R^{n+k}$ is called $(\Lambda, \alpha, r_0)$-almost minimizing in $U$ if
\begin{gather}
||T||(W) \leq ||T+\del R||(W) + \Lambda r^{n+\alpha} 
\end{gather}
for $x \in \spt T$, $W \subset\subset B_r(x) \subset U$ with $r < r_0$, and all integral $(n+1)$-currents $R$ supported in $W$. The same definition carries over in the obvious way to integral $n$-currents in a Riemannian manifold $N^{n+k}$.

Analogous to varifolds with bounded mean-curvature, there is a monotonicity formula \eqref{eqn:am-mono}, an Allard-type regularity theorem \cite{Bo}, and a compactness theorem (Lemma \ref{lem:am-compact}) for almost-minimizing currents.  Recall that the density ratios of $T$ are given by
\[
\theta_T(x, r) := \frac{||T||(B_r(x))}{\omega_n r^n} .
\]
Following the computations of \cite{DeSpSp}, we have the inequality
\begin{equation}\label{eqn:am-mono}
\left[\theta_T(x, r) + \frac{n\Lambda}{\omega_n \alpha} r^\alpha\right] - \left[ \theta_T(x, s) + \frac{n\Lambda}{\omega_n \alpha} s^\alpha\right] \geq \frac{1}{2\omega_n} \int_{B_r(x) \setminus B_s(x)} \frac{|\pi_T^\perp(y - x)|^2}{|y - x|^{n+2}} d||T||(y)
\end{equation}
for all $B_s(x) \subset B_r(x) \subset U$.  In particular $r \mapsto \theta_T(x, r) + \frac{n\Lambda}{\omega_n \alpha} r^\alpha$ is increasing.

Let us also remark that the general compactness/regularity theory of almost-minimizing currents in $N^{n+k}$ is essentially the same as in $\R^{n+k}$.  For, if $g$ is a $C^2$ Riemannian metric on $B_1 \subset \R^{n+k}$, satisfying $|g - \geucl|_{C^2} \leq \delta$, then (provided $\delta(n)$ is sufficiently small) we have $(1-\delta)|x - y| \leq d_g(x, y) \leq (1+\delta)|x - y|$, and $B_{(1-\delta)r}(x) \subset B^g_r(x) \subset B_{(1+\delta)r}(x)$, and for every $B_r(x) \subset B_{1-10\delta}$ there is a normal change of coordinates $\phi_x : B^g_{(1+\delta)r}(x) \to B^g_{(1+\delta)r}(x)$, in which the metric satisfies $|(\phi_x^* g)(z) - \geucl| \leq c(n) \delta |z - x|^2$.  It's straightforward to verify that if $T$ is $(\Lambda, \alpha, r_0)$-almost-minimizing in $(B_1, g)$, then $T$ is $(\Lambda + C \delta , \min\{ \alpha, 2\}, (1-\delta)r_0)$-almost-minimizing in $(B_{1-10\delta}, \geucl)$, for $C$ a constant depending on $n, \Lambda, ||T||(B_1)$.

\vspace{3mm}

A log-epiperimetric inequality for smooth, minimizing cones was established by \cite{EnSpVe2}, who used it to prove uniqueness of smooth, multiplicity-one tangent cones for almost-minimizing currents.  We outline here the symmetric log-epiperimetric inequality for smooth cones -- the details are essentially the same as in Section \ref{sec:ac}.

Take $\bC^n \subset \R^{n+k}$ a minimal cone with smooth cross section $\Sigma$.  If $u : \Sigma \to \Sigma^\perp$ or $h : \bC \cap A_{R, \rho} \to \bC^\perp$, define the spherical graphing functions
\begin{align}
G_\Sigma(u) &= \left\{ \frac{ \theta + u(\theta)}{|\theta + u(\theta)|} : \theta \in \Sigma \right\}, \\
G_\bC(h) \cap A_{R,\rho} &= \left\{ |x| \frac{ x + h(x)}{|x + h(x)|} : x \in \bC \cap A_{R, \rho} \right\} .
\end{align}
For ease of notation write $G_\bC(h) = G_\bC(h) \cap A_{1, 0}$.  Recall that we defined $A_{R, \rho}(x) = B_R(x) \setminus \overline{B_\rho}(x)$, and $A_{R, 0}(x) = B_R(x) \setminus \{0\}$.

By a straightforward adaption of the proof in Section \ref{sec:ac}, one can prove the following symmetric log-epiperimetric inequality.  We remark that the epiperimetric inequality as stated in \cite{EnSpVe2} used the $C^{1,\alpha}$ norm, but it is not hard to see that it suffices to consider only the $C^1$ norm.
\begin{theorem}[Symmetric log-epiperimetric inequality for smooth minimal cones]\label{thm:am-epi}
Let $\bC^n \subset \R^{n+k}$ (for $n \geq 2$) be a smooth, minimal cone, with cross section $\Sigma = \bC \cap \del B_1$.  There exist $\eps(\bC, k), \delta(\bC, k)$ positive, and a $\gamma(\bC, k) \in [0, 1)$ so that the following holds.

Let $u : \Sigma \to \Sigma^\perp$ be such that $|u|_{C^1} < \delta$.  Then there is a function $h \in C^1(\bC \cap B_1 \setminus \{0\}, \bC^\perp)$ satisfying
\begin{gather}\label{eqn:epi-concl1}
h|_{\bC \cap \del B_1} = u, \quad |x|^{-1} |h| + |\nabla h| \leq c(\Sigma, k) |u|_{C^1}^{(1+\gamma)/2},
\end{gather}
so that
\begin{gather}\label{eqn:epi-concl2}
\haus^n(G_\bC(h)) \leq \haus^n(G_\bC(z)) - \eps|\haus^n(G_\bC(z)) - \haus^n(G_\bC(0))|^{1+\gamma},
\end{gather}
where $z(x) = |x| u(x/|x|)$ is the $1$-homogenous extension of $u$.  If $\bC$ is integrable,\footnote{As in \cite{EnSpVe2}, integrable means that every $1$-homogenous Jacobi field on $\bC$ can be realized by a $1$-parameter family of smooth, minimal cones.} then one can take $\gamma = 0$.
\end{theorem}

Let $T$ be a $(\Lambda, \alpha, 1)$-almost-minimizing $n$-current in $B_1 \subset \R^{n+k}$ satisfying $\del T = 0$ and 
\[
T \llcorner A_{1, \rho} = [G_\bC(u) \cap A_{1, \rho}], \quad |x|^{-1} |u| + |\nabla u| \leq \delta,
\]
for $\bC$ a smooth minimizing cone in $\R^{n+k}$ and some $\rho \in [0, 1/2]$.  Provided $\delta(\bC, k)$ is sufficiently small, the coarea formula, the computation \cite[(2.5)]{DeSpSp}, and the Hardt-Simon inequality (see e.g. \cite[(11)]{Simon1}) imply that
\begin{equation}
\frac{d}{dr} \theta_T(0, r) \geq \frac{n}{2r} (\theta_{T_r}(0, r) - \theta_{T}(0, r)) + \frac{1}{4\omega_n r} \int_\Sigma r^2|\del_r u_r|^2 - \Lambda r^{\alpha - 1} \quad \forall r \in (\rho, 1),
\end{equation}
where $T_r$ is the cone over $\del (T \llcorner B_r)$, and $u_r(x) = r^{-1} u(rx)$.

On the other hand, from Theorem \ref{thm:am-epi} we get the symmetric epiperimetric inequality
\begin{equation}\label{eqn:am-epi}
\theta_T(0, r) \leq \theta_{T_r}(0, r) - \eps |\theta_{T_r}(0, r) - \theta_\bC(0)|^{1+\gamma} + \Lambda r^\alpha \quad \forall r \in (\rho, 1) ,
\end{equation}
where $T_r$ denotes the cone over $\del (T \llcorner B_r)$.

Ensuring $\Lambda, |\theta_T(0, 1) - \theta_\bC(0)|, |\theta_T(0, \rho) - \theta_\bC(0)|$ are sufficiently small (depending only on $n, \alpha$), we have by \eqref{eqn:ac-mono} that
\[
|\theta_T(0, r) - \theta_\bC(0)| \leq 1, \quad \forall r \in (\rho, 1).
\]
One can then prove an exact analogue of Theorem \ref{thm:main}, with $\Sigma$ in place of $\del B_1$ and $\theta_T(0, r) - \theta_\bC(0)$ in place of $\cE(u_r)$ and $\theta_{T_r}(0, r) - \theta_\bC(0)$ in place of $\cE(z_r)$, to deduce the function
\[
G(r) = \theta_T(0, r) - \theta_\bC(0) + 3\alpha^{-1} \Lambda r^\alpha
\]
satisfies $G' \geq (\delta'/r) |G|^{1+\gamma}$ as in \eqref{eqn:main-concl1}, and to deduce the Dini estimate
\begin{align}
&\int_{\bC \cap A_{1, \rho}} |\del_r (u/r)| r^{1-n} \nonumber \\
&\leq c(\bC, \alpha)( |\theta_T(0, 1) - \theta_\bC(0)|^{(1-\gamma)/2} + |\theta_T(0, \rho) - \theta_\bC(0)|^{(1-\gamma)/2} + \Lambda^{(1-\gamma)/2}). \label{eqn:am-dini}
\end{align}
A direct application is the following analogue of Theorem \ref{thm:ac-param} for almost-minimizing currents, which says that graphicality propagates as long as the density stays close to the density of the cone (c.f. \cite[Theorem 13.1]{Sim}).
\begin{theorem}\label{thm:am-param}
Let $\bC^n \subset \R^{n+k}$ be a smooth, minimal cone.  Given $\eps, \alpha > 0$, there are $\Lambda(\bC, \eps, \alpha, k)$, $\delta(\bC, \eps, \alpha, k)$ positive so that the following holds.

Let $T$ be $(\Lambda, \alpha, 1)$-almost-minimizing in $B_1 \subset \R^{n+k}$ with $\del T = 0$, and let $\rho \geq 0$, such that
\begin{gather}
\spt T \cap A_{1, 1/2} = G_\bC(u) \cap A_{1, 1/2}, \quad |u|_{C^1} \leq \delta \\
\theta_T(0, 1) \leq \theta_\bC(0) + \delta, \quad \theta_T(0, \rho/2) \geq \theta_\bC(0) - \delta.
\end{gather}
Then $\spt T \cap A_{1, \rho} = G_\bC(u) \cap A_{1,\rho}$, with $|x|^{-1} |u| + |\nabla u| \leq \eps$.  Moreover, $u$ admits the Dini-type estimate
\begin{gather}
\int_{\bC \cap A_{1, \rho}} |\del_r (u/r)| r^{1-n} \leq \eps.
\end{gather}
\end{theorem}

\begin{proof}
Same as Theorem \ref{thm:ac-param}, using the Dini estimate \eqref{eqn:am-dini}, and Lemma \ref{lem:am-extend} in place of Lemma \ref{lem:ac-extend}.
\end{proof}

\begin{lemma}\label{lem:am-extend}
Let $\bC^n \subset \R^{n+k}$ be a smooth minimal cone.  Given $\eps, \alpha > 0$, there are $\Lambda(\bC, \eps, \alpha, k)$, $\delta(\bC, \eps, \alpha, k)$ positive so that the following holds.

Let $T$ be $(\Lambda, \alpha, 1)$-minimizing integral $n$-current in $T$ with $\del T = 0$, and suppose that
\begin{gather}
\spt T \cap A_{1, 1/2} = G_\bC(u) \cap A_{1, 1/2}, \quad |u|_{C^1} \leq \delta, \\
\theta_T(0, 1) \leq \theta_\bC(0) + \delta, \quad \theta_T(0, 1/8) \geq \theta_\bC(0) - \delta.
\end{gather}
Then $\spt T \cap A_{1, 1/4} = G_\bC(u) \cap A_{1, 1/4}$, with $|u|_{C^1} \leq \eps$.
\end{lemma}

\begin{proof}
Follows by a straightforward contradiction argument, using Lemma \ref{lem:am-compact}, \cite{Bo} and taking $\Lambda, \delta \to 0$.
\end{proof}

Since the other ingredients involved (monotonicity, compactness, partial regularity) already have direct analogues for the case of almost-minimizers, most of the results in \cite{Ed} carry over to almost-minimizers.  For example, we have:
\begin{theorem}[Almost-minimizers near Simons' cones]\label{thm:am-lious}
Let $\bC^n \subset \R^{n+1}$ be a minimizing quadratic hypercone, and let $\{ S_\lambda\}_\lambda$ be the associated Hardt-Simon foliation (see \cite{EdSp} for notation).  Given $\eps, \alpha > 0$, there is a $\delta(\bC, \eps, \alpha) > 0$ so that the following holds.

Let $T$ be a $(\delta, \alpha, 1)$-almost-minimizing $n$-current in $B_1$ with $\del T = 0$, and suppose that
\begin{equation}\label{eqn:am-simons-hyp}
d_H(\spt T \cap B_1, \bC \cap B_1) \leq \delta, \quad (1/2)\theta_\bC(0) \leq \theta_T(0, 1/2), \quad \theta_T(0, 1) \leq (3/2) \theta_\bC(0).
\end{equation}
Then we can find an $a \in \R^{n+1}$, $\lambda \in \R$, $q \in SO(n+1)$, satisfying
\begin{equation}\label{eqn:am-simons-concl1}
|a| + |q - Id| + |\lambda| \leq \eps,
\end{equation}
and a $C^{1}$ function $u : (a + q(S_\lambda)) \cap B_{1/2}(a) \to S_\lambda^\perp$, so that
\begin{equation}\label{eqn:am-simons-concl2}
\spt T \cap B_{1/4} = \graph_{a + q(S_\lambda)}(u) \cap B_{1/4}, \quad |x - a|^{-1} |u| + |\nabla u| \leq \eps.
\end{equation}

If $\bC^n \subset \R^{n+1}$ is a general smooth (away from $0$) area-minimizing hypercone, and one additionally assumes $\spt T$ lies to one side of $\bC$, then the same conclusion holds with $a = 0, q = Id$.  Moreover, if $T$ is in fact mass-minimizing, then either $\spt T = \bC \cap B_1$ or $\lambda \neq 0$.
\end{theorem}

\begin{proof}
Same as in \cite[Theorem 2.8]{Ed}, except with Theorem \ref{thm:am-param} in place of \cite[Theorem 13.1]{Ed}, and Lemma \ref{lem:am-compact}, \cite{Bo} in place of Allard's varifold compactness, regularity.  If $\spt T$ lies to one side of $\bC$, then since $T$ is assumed to be (almost-)minimizing, one can use the Liouville theorem of \cite{HaSi} in place of \cite{simon:liousville}, and thereby assume only that $\bC$ is minimizing.  The very last statment, that $\spt T = \bC$ when $T$ is minimizing and $\lambda = 0$, follows from \eqref{eqn:am-simons-concl2} and the strong maximum principle \cite{Si:max}.
\end{proof}

\begin{theorem}[Finite diffeotypes of almost-minimizers]
Let $(N^{n+1}, g)$ be a closed Riemannian manifold of dimension $n+1 \leq 8$.  Given any $\Lambda, \Gamma \geq 0$, $\alpha, r > 0$, there is a constant $C(N, g, \Lambda, \alpha, r, \Gamma)$ so if $T$ is any $(\Lambda, \alpha, r)$-almost-minimizing $n$-current in $(N, g)$, with $\del T = 0$ and having mass $||T||(N) \leq \Gamma$, then $\sing T$ is discrete (if $n = 7$)/ empty (if $n \leq 6$), $\reg T$ fits into one of $C$ diffeomorphism classes, and $\spt T$ fits into one of $C$ bi-Lipschitz equivalence classes.
\end{theorem}

\begin{proof}
Same as \cite[Theorem 2.4]{Ed}.
\end{proof}

The following compactness theorem for almost-minimizing currents should be standard, but we were not able to find a reference.
\begin{lemma}\label{lem:am-compact}
Let $T_i$ be a sequence of $(\Lambda, \alpha, r_0)$-almost-minimizing $n$-currents in $U \subset \R^{n+k}$.  Suppose that
\[
\sup_i ||T_i||(W) + ||\del T_i||(W) < \infty \quad \forall W \subset\subset U.
\]

Then after passing to a subsequence, we can find a $(\Lambda, \alpha, r_0)$-almost-minimizing $n$-current $T$ so that $T_i \to T$ as currents and $||T_i|| \to ||T||$ as Radon measures.
\end{lemma}

\begin{proof}
The proof that $T$ is $(\Lambda, \alpha, r_0)$-almost-minimizing is the same as the proof for $T_i, T$ being mass-minimizing (see e.g. \cite[Chapter 7, Theorem 2.4]{Sim}).  We highlight here how to show convergence $||T_i|| \to ||T||$.  Passing to a subsequence $i'$, we can assume that $||T_{i'}|| \to \lambda$, for some Radon measure $\lambda$.  Lower-semi-continuity of mass implies $||T||(W) \leq \lambda(W)$ for all $W \subset\subset U$.

Given $x \in \spt T \cap U$, then by the monotonicity formula \eqref{eqn:am-mono} we have $\theta_T(B_r(x)) \geq 1/2$ for all $r$ sufficiently small.  Arguing as in \cite{Sim}, we have for a.e. $r$ small the inequality
\[
\lambda(B_r(x)) \leq ||T||(B_r(x)) + \Lambda r^{n+\alpha} \leq (1+c(n)\Lambda r^\alpha)||T||(B_r(x)),
\]
which in turn implies 
\[
1 = \lim_{r \to 0} \frac{||T||(\overline{B_r(x)})}{\lambda(\overline{B_r(x)})}
\]
for $\lambda$-a.e. $x$.  Since $||T|| << \lambda$, the Radon-Nikodyn theorem implies $\lambda = ||T||$.
\end{proof}

\subsection{Currents in an annulus}\label{ssec:am}

Lastly we comment on integral $n$-currents $T$ in $A_{1, \rho} \subset \R^{n+k}$, for some $\rho \in (0, 1/2]$, with $\del T = 0$ in $A_{1, \rho}$ and which are $(\Lambda, \alpha, 1)$-almost-minimizing in $A_{1, \rho}$.  $T$ need only be defined in $A_{1, \rho}$, and we shall interpret $||T||(B_r) \equiv ||T||(A_{r, \rho})$.  In particular $\theta_T(0, r) \equiv \frac{||T||(A_{r, \rho})}{\omega_n r^n}$.

We shall additionally assume $T$ satisfies the following ``global'' almost-minimizing property:
\begin{gather}\label{eqn:am-ann-cond}
\left\{ \begin{array}{l} ||T||(B_r) \leq ||S||(B_r) + \Lambda r^{n+\alpha} \text{ for any integral $n$-current $S$ in $B_1$} \\
\text{ with $\del S = 0$ in $B_1$ and $\spt (S - T) \subset \overline{A_{r, \rho}}$} \end{array} \right.
\end{gather}
The comparison principle \eqref{eqn:am-ann-cond} arises naturally in obstacle problems like in \cite{MaNo}.

By standard slicing theory \cite[Section 6.4]{Sim}, for a.e. $r \in (\rho, 1)$ the cone $T_r$ over $\del (T \llcorner A_{r, \rho})$ (taking the boundary in $A_{1, \rho}$) is an integral $n$-current in $B_r$ with $\del T_r = \del (T \llcorner A_{r, \rho})$.  By the deformation lemma \cite[Section 6.5]{Sim}, we can find an integral $n$-current $S_r$ supported in $\del B_\rho$ with $\del S_r = \del (T_r \llcorner B_\rho)$ and $||S_r||(\del B_\rho) \leq c_0(n, k) ||T_r||(B_\rho) = c_0 \rho^n r^{-n} ||T_r||(B_r)$.

From assumption \eqref{eqn:am-ann-cond} we deduce that for a.e. $r \in (M\rho, 1)$ we have
\[
||T||(B_r) \leq (1+c_0 \rho^n r^{-n}) ||T_r||(B_r) + \Lambda r^{n+\alpha} \leq 2 ||T_r||(B_r) + \Lambda r^{n+\alpha}
\]
provided we ensure $M(n, k)$ is suficiently large.  By repeating a similar computation to \eqref{eqn:am-mono}, we deduce that
\begin{align*}
&\left[ \frac{\theta_T(0, r)}{1+c_0 \rho^n r^{-n}} + \frac{n\Lambda}{\omega_n \alpha} r^\alpha \right] - \left[ \frac{\theta_T(0, s)}{1+c_0 \rho^n s^{-n}} + \frac{n\Lambda}{\omega_n \alpha} r^\alpha \right] 
&\geq \frac{1}{4\omega_n} \int_{A_{r, s}} \frac{|\pi_T^\perp(x)|^2}{|x|^{n+2}} d||T||(x).
\end{align*}
for all $M \rho < s < r < 1$.  In particular, given any $\eps > 0$ then provided $M(n, k, \alpha, \eps)^{-1}, \Lambda(n, k, \alpha, \eps)$ are sufficiently small, we have
\begin{equation}\label{eqn:am-ann-mono-eps}
\theta_T(0, M\rho) - \eps \leq \theta_T(0, r) \leq \theta_T(0, 1) + \eps \quad \forall r \in (M\rho, 1).
\end{equation}

\vspace{3mm}

Suppose $T$ additionally satisfies
\[
T \llcorner A_{1, M \rho} = [G_\bC(u) \cap A_{1, M\rho}], \quad |x|^{-1} |u| + |\nabla u| < \delta.
\]
Now for every $r \in (M \rho, 1)$, we note that $||T_r||(B_r) \leq c_1(\bC, k) r^n$ provided $\delta(\bC, k)$ is sufficiently small, and therefore we can argue as in the previous section to deduce the monotonicity
\begin{align*}
\frac{d}{dr} \theta_{T}(0, r) 
&\geq \frac{n}{2r} (\theta_{T_r}(0, r) - \theta_{T}(0, r)) + \frac{1}{4\omega_n r} \int_\Sigma r^2|\del_r u_r|^2 - c_1 \rho^n r^{-n-1} - \Lambda r^{\alpha - 1} \\
&\geq \frac{n}{2r} (\theta_{T_r}(0, r) - \theta_{T}(0, r)) + \frac{1}{4\omega_n r} \int_\Sigma r^2|\del_r u_r|^2 - c_1 M^{\alpha-n} \rho^\alpha r^{-\alpha-1} - \Lambda r^{\alpha - 1} 
\end{align*}
and epiperimetric inequality
\begin{equation*}
\theta_{T}(0, r) \leq \theta_{T_r}(0, r) - \eps |\theta_{T_r}(0, r) - \theta_\bC(0)|^{1+\gamma} + c_1 M^{\alpha-n} \rho^\alpha r^{-\alpha} + \Lambda r^\alpha .
\end{equation*}

Noting that $c_1 M^{\alpha-n} \rho^\alpha r^{-\alpha-1} \leq c_1 M^{-n}$, then provided $M^{-1}(\bC, \alpha, k)$, $\Lambda(\bC, \alpha, k)$ are sufficiently small and $\theta_T(0, 1) \leq \theta_\bC(0) + 1/2$, $\theta_T(0, M\rho) \geq \theta_\bC(0) - 1/2$, we can prove the analogue of Theorem \ref{thm:main} with $\del B_1 \equiv \Sigma$, $\cE(u_r) \equiv \theta_{T, \rho}(0, r) - \theta_\bC(0)$, $\cE(z_r) \equiv \theta_{T_r, \rho}(0, r) - \theta_\bC(0)$, $\Lambda_+ = \Lambda$, $\Lambda_- = c_1M^{\alpha-n} \rho^\alpha$, to deduce that the function
\[
G(r) = \theta_{T}(0, r) - \theta_\bC(0) + 3\alpha^{-1} \Lambda r^\alpha - 3\alpha^{-1} c_1 M^{\alpha-n}\rho^{\alpha} r^{-\alpha}
\]
satisfies the ODE $G' \geq (\delta'/r) |G|^{1+\gamma}$ as in \eqref{eqn:main-concl1}, and to deduce the Dini estimate
\begin{align}
\int_{\bC \cap A_{1, M\rho}} |\del_r (u/r)| r^{1-n} 
&\leq c(\bC, k, \alpha)( | \theta_{T, \rho}(0, M\rho) - \theta_\bC(0)|^{(1-\gamma)/2} + |\theta_{T, \rho}(0, 1) - \theta_\bC(0)|^{(1-\gamma)/2} ) \nonumber \\
&\quad + c(\bC,k,  \alpha) (\Lambda^{(1-\gamma)/2} + M^{-n(1-\gamma)/2} ). \label{eqn:am-ann-dini}
\end{align}

We obtain the following variant of Theorem \ref{thm:am-param}.  Theorem \ref{thm:am-ann-param} below is a ``minimizing'' version of \cite[Theorem 6.3]{Ed}, and gives an alternate approach to the ``mesoscale flatness'' of \cite[Theorem 1.9]{MaNo}.
\begin{theorem}\label{thm:am-ann-param}
Let $\bC^n \subset \R^{n+k}$ be a smooth minimal cone.  Given $\eps, \alpha > 0$ there is a $\delta(\bC, k, \eps, \alpha)$ positive so that the following holds.  Let $\rho \geq 0$, and let $T$ be $(\delta, \alpha, 1)$-almost-minimizing in $A_{1, \delta \rho}$ with $\del T = 0$, and satisfying the global comparison property \eqref{eqn:am-ann-cond}, and for which
\begin{gather*}
T \llcorner A_{1, 1/2} = [G_\bC(u) \cap A_{1, 1/2}], \quad |u|_{C^1} \leq \delta \\
\theta_T(0, 1) \leq \theta_\bC(0) + \delta, \quad \theta_T(0, \rho/2) \geq \theta_\bC(0) - \delta.
\end{gather*}
Then $T \llcorner A_{1, \rho} = [G_\bC(u) \cap A_{1, \rho}]$ with $|x|^{-1} |u| + |\nabla u| \leq \eps$.  Morevoer, $u$ admits the Dini estimate
\[
\int_{\bC \cap A_{1, \rho}} |\del_r (u/r)|r^{1-n} \leq \eps.
\]
\end{theorem}

\begin{proof}
Same as Theorem \ref{thm:am-param}, using Lemma \ref{lem:am-ann-extend}, \eqref{eqn:am-ann-mono-eps}, \eqref{eqn:am-ann-dini}.
\end{proof}

\begin{lemma}\label{lem:am-ann-extend}
Let $\bC^n \subset \R^{n+k}$ be a smooth minimal cone.  Given $\eps, \alpha > 0$, there is a $\delta(\bC, \eps, \alpha, k)$ positive so that the following holds.  Let $T$ be $(\delta, \alpha, 1)$-almost-minimizing integral $n$-current in $A_{1, \delta}$ with $\del T = 0$, and suppose that
\begin{gather*}
T \llcorner A_{1, 1/2} = [G_\bC(u) \cap A_{1, 1/2}], \quad |u|_{C^1} \leq \delta, \\
\theta_T(0, 1) \leq \theta_\bC(0) + \delta, \quad \theta_T(0, 1/8) \geq \theta_\bC(0) - \delta.
\end{gather*}
Then
\[
T \llcorner A_{1, 1/4} = [G_\bC(u) \cap A_{1, 1/4}], \quad |u|_{C^1} \leq \eps.
\]
\end{lemma}

\begin{proof}
Same as Lemma \ref{lem:am-extend}.
\end{proof}

\end{document}